\documentclass{amsart}

\usepackage{amssymb}
\usepackage[usenames, dvipsnames]{color}
\usepackage{amsmath,latexsym,amsthm,amsfonts}
\usepackage{graphicx,amsfonts}
\usepackage{hyperref}

\usepackage{enumitem, hyperref}
\makeatletter
\def\namedlabel#1#2{\begingroup
    #2%
    \def\@currentlabel{#2}%
    \phantomsection\label{#1}\endgroup
}
\makeatother

\newtheorem{remark}{Remark}[section]
\newtheorem{definition}{Definition}[section]
\newtheorem{assumption}{Assumption}[section]

\newtheorem{theorem}{Theorem}[section]
\newtheorem{lemma}[theorem]{Lemma}
\newtheorem{proposition}[theorem]{Proposition}
\newtheorem{corollary}[theorem]{Corollary}

\def\be{\begin{equation}}
\def\ee{\end{equation}}
\def\cR{\mathbb{R}}

\def\intO{\int_\Omega}
\def\intT{\int_0^T}
\def\spaceH{L^2(\Omega)}
\def\spaceV{H^1(\Omega)}
\def\spaceVp{H^1(\Omega)^*}
\def\Uad{\mathcal{U}_{\rm ad}}
\def\Uinfty{L^\infty(0,T;\cR^n)}
\def\Utwo{L^2(0,T;\cR^n)}
\def\ub{\bar{u}}
\def\pb{\bar{p}}
\def\rhob{\bar{\rho}}
\def\eps{\varepsilon}
\def\div{{\rm div}}
\def\ds{\displaystyle}

\usepackage[ulem=normalem,draft]{changes}

\newcommand{\tcr}{\textcolor{red}}

\newcommand{\fredi}{\textcolor{cyan}}

\title[Optimal Control of the Fokker-Planck equation]{First and Second Order Optimality Conditions for the Control of Fokker-Planck Equations$^*$}

\author[M.S. Aronna]{M. Soledad Aronna}
\address{M.S. Aronna,  Escola de Matem\'atica Aplicada, Funda\c{c}\~ ao Get\'ulio Vargas,  Rio de Janeiro 22250-900, Brazil
 }
\email{soledad.aronna@fgv.br}

\author[F. Tr\"oltzsch]{Fredi Tr\"oltzsch}
\address{F. Tr\"oltzsch, Institut f\"ur Mathematik, Technische Universit\"at Berlin, Stra{\ss}e des 17. Juni 136, D-10623 Berlin, Germany}
\email{troeltz@math.tu-berlin.de}

\thanks{$^*$This is a longer version of the article with the same title that will appear in ESAIM: Control, Optimisation and Calculus of Variations. The differences are in Subsection 6.5 on Second order sufficient conditions and are detailed at the beginning of page 21. \\
The first author was supported by  CAPES (Brazil) and by the Alexander von Humboldt Foundation (Germany).}

\begin{document}

\maketitle

\begin{flushright}
{\em \today}
\end{flushright}

\begin{abstract}
In this article we study an optimal control problem subject to the Fokker-Planck equation
\[
\partial_t \rho - \nu \Delta \rho - {\rm div } \big(\rho B[u]\big) = 0.
\] 
The control variable $u$ is time-dependent and possibly multidimensional, and the function $B$ depends on the space variable and the control. The cost functional is of tracking type and includes a quadratic regularization term on the control. For this problem, we prove existence of optimal controls and  first order necessary  conditions. Main emphasis is placed on second order necessary and sufficient conditions.
\end{abstract}

\section{Introduction}

In this article we prove first and second order optimality conditions for a control problem subject to the Fokker-Planck equation
\be
\label{FPintro}
\partial_t \rho - \nu \Delta \rho - {\rm div } \big(\rho B[u]\big) = 0,
\ee
where $u$ is a time-dependent possibly multidimensional control, and $B[u](x)=c(x) + b(x)\otimes u$ is defined by given vector functions $c$ and $b$.

 Fokker-Planck equations arise in many situations in which a large number of agents is involved. More precisely, these equations are known to describe the time evolution of the probability density function of agents, where the motion of each of them is modelled by a stochastic differential equation.
 In particular, Fokker-Planck equations are present in models of mass opinion dynamics \cite{AlbiPareschiZanella2014}, tumor growth \cite{BoseTrimper2009,Herty2012}, bird flocks movement \cite{DuanFornasierToscani2010} and various biological events \cite{SchienbeinGruler1993,Chavanis2008nonlinear}, among others. The recent survey \cite{Furioli2017} describes several applications of the Fokker-Planck equation to different socio-economic phenomena. It is worth mentioning that some of these articles already analyze control problems associated to their models, as {\em e.g.}  \cite{Herty2012,AlbiPareschiZanella2014}.

Another motivation for investigating  optimal control problems governed by Fok\-ker-Plack equations comes from Mean Field Games (MFG) theory \cite{LasryLions2007,GomesSaude2014}. 
%\deleted{MFG' systems consist in a Hamilton-Jacobi equation coupled with a Fokker-Planck one.} 
It has been observed \cite[Section 4]{LasryLions2007} (see also \cite{cardaliaguet2010notes,Ryzhik2018}) that, under a specific choice of the objective functional, the MFG system can be interpreted as a first order optimality system for the control of a Fokker-Planck equation.

\if{Optimal control problems governed by  \eqref{FPintro} have been recently studied in a few articles.  A problem with one-dimensional piecewise constant controls is treated in \cite{AnnunziatoBorzi2010}, where the authors propose control strategies based on a receding horizon model predictive control framework. 
That work was extended later in \cite{AnnunziatoBorzi2013} to the case of multidimensional constant control. A Fokker-Planck equation with a space-dependent time-independent control is investigated in \cite{RoyAnnunziatoBorzi2016fokker}. There, the authors prove existence of an optimal control,  formulate a first-order optimality system, and propose a numerical scheme along with simulations. These three last article focus mainly on numerics. We also refer the reader to the survey \cite{AnnunziatoBorzi2018} for a review on Fokker-Planck control frameworks. {\em Fredi: I added these two sentences. If you think it is excessive, please cut what you believe is reasonable.}}}\fi

For a review on Fokker-Planck control frameworks, we refer the reader to the survey \cite{AnnunziatoBorzi2018} and to the references therein. 
Optimal control problems governed by  \eqref{FPintro} have been recently studied in  quite a number of articles. 
For numerical control strategies, we mention in particular \cite{AnnunziatoBorzi2010} and \cite{AnnunziatoBorzi2013}, where piecewise constant 
one-dimensional and multidimensional constant controls, respectively, are discussed. The case of space-dependent time-independent control is investigated in \cite{RoyAnnunziatoBorzi2016fokker}.
Moreover, we refer to the following papers that are closer to our control problem and mainly concentrate on aspects of the analysis.  
For a more general setting with time- and space-dependent control, existence results and first order optimality conditions are proved in \cite{FleigGuglielmi2017} and \cite{AlbiChoiFornasierKalise2017}. Their results on  first order analysis are similar to ours.
We also mention  \cite{BreitenKunischPfeiffer2018}, where the authors show results on stabilization of \eqref{FPintro} through linearization (more details in Remark \ref{RemarkBiblio}).

Optimization problems associated to \eqref{FPintro} belong to the class of {\em bilinear optimal control}. This framework has been considered in {\em e.g.} \cite{CasasWachsmuthWachsmuth2018} for elliptic equations. On the other hand, the work \cite{AronnaBonnansKroener2018} dealt with infinite dimensional bilinear dynamical systems, but their results do not apply here. 

In this paper we provide first and second order optimality conditions for an optimal control problem associated to \eqref{FPintro}, under quite mild regularity assumptions on the data functions and the spatial domain
of the state equation. The cost functional is of tracking type and includes a {\em quadratic regularizing term} on the control.  The control is only time-dependent. We obtain our second order necessary and sufficient conditions by application of  results proved in \cite{CasasTroeltzsch2012} in an abstract framework.

Our paper is organized as follows. In Section \ref{S2}, we present the equation, the basic assumptions, show well-posedness and some other properties of the state equation. The optimal control problem is introduced in Section \ref{S3}, where existence of optimal solution is proved. In Section \ref{S4}, we discuss properties of the control-to-state mapping, while Section \ref{S5} presents the adjoint system and first order optimality conditions. 
Section \ref{S6} is devoted to second order analysis and contains our main theorems. Finally, in the Appendix \ref{appendix} we included some proofs of auxiliary results.

%%%%%%%%%%%%%%%%%%%%%%%%%%%%%%%%%%%%%%%%%%%%%%%%%%%%%%%
\subsection*{Notation} \label{S1}
%%%%%%%%%%%%%%%%%%%%%%%%%%%%%%%%%%%%%%%%%%%%%%%%%%%%%%%
Given a real interval $[0,T]$, a normed space $X$ and $p\in [1,\infty],$  we let $L^p(0,T;X)$ denote the Lebesgue space of $L^p$-functions and  write $C([0,T];X)$ for the space of continuous functions, both with domain $[0,T]$ and values in $X.$
When $X=\cR$ we just write $L^p(0,T)$ and, for any $m$, we let $\|\cdot\|_{p}$ denote the norm in $L^p(0,T;\cR^m)$.
 Analogously, for a set $\Omega \subseteq \cR^n,$ $L^p(\Omega;\cR^m)$ is the Lebesgue space of $L^p$-functions with domain $\Omega$ and values in $\cR^m$.  We omit $\cR^m$ when $m=1$ and the values range in $\cR,$  and we  use the  form $\|\cdot\|_{L^p(\Omega)^m}$ to denote the norm in $L^p(\Omega;\cR^n)$. 

Throughout the article, we consider the real Hilbert spaces $\spaceH$, $\spaceV$ and $\spaceVp.$ We let $(\cdot,\cdot)$ denote the scalar product in $\spaceH$ and  $\langle \cdot,\cdot\rangle$ the dual pairing between $\spaceVp$ and $\spaceV$. Other scalar products and pairings will be distinguished by specifying the spaces as subindexes. For two vectors $u, v \in \mathbb{R}^n$, the result of the componentwise multiplication $u\otimes v$ is defined by 
the vector $ w \in \mathbb{R}^n$ with components $w_i = u_i v_i$, for $i = 1,\ldots,n$.

%%%%%%%%%%%%%%%%%%%%%%%%%%%%%%
\section{The controlled Fokker-Planck equation} \label{S2}
%%%%%%%%%%%%%%%%%%%%%%%%%%%%%%

We consider the Fokker-Planck equation with initial and boundary conditions given by
\begin{align}
\label{FP}
\partial_t \rho(x,t) - \nu \Delta \rho(x,t) - {\rm div } \big(\rho(x,t) B[u(t)](x)\big) &= 0\quad \text{in } Q,\\
\label{FPinit} \rho(x,0) &= \rho_0(x)\quad \text{in } \Omega,\\
\label{PBNeumann}\Big(\nu \nabla \rho(x,t) + \rho(x,t) B[u(t)](x)\Big)\cdot n(x) &=0\quad \text{on } \Sigma,
\end{align}
where $\nu>0,$ $\rho_0 \in \spaceH,$
$\Omega \subset \cR^n$ is a bounded domain with Lipschitz boundary $\Gamma:= \partial \Omega,$ and we set $\Sigma := \Gamma \times (0,T),$ $Q:=\Omega \times (0,T).$
The control is $u=(u_1,\dots,u_{n}) \in L^\infty(0,T;\cR^{n})$, 
and the function $B \colon \cR^n \times \Omega \to \cR^n$ is given by
\[
{B}[u](x) := c(x) + {b}(x)\otimes {u},
\]
where $c,b \in L^\infty(\Omega;\cR^n)$ are fixed.
In \eqref{FP}, the differential operators $\Delta$ and ${\div }$ act only with respect to the spatial coordinate $x$. For definition and basic properties of the $\div$-operator we refer to \cite{GiraultRaviart2012}; in particular, we frequently use the Green's formula \cite[(2.17), p. 28]{GiraultRaviart2012}. In \eqref{PBNeumann}, $n$ denotes the outward normal unit vector on $\Gamma$.

\begin{remark}
\label{RemarkBiblio} Let us compare this equation with others considered in the literature. 

Breiten, Kunisch and Pfeiffer \cite{BreitenKunischPfeiffer2018} assume $B$ to take the form $\nabla_x V$ for a potential $V$ given by
\[
V(x,t) = G(x) + \gamma(x) u(t)
\]
for a scalar control $u$  and  study the infinite horizon stabilization problem. 
In \cite{Breiten_Kunisch_Pfeiffer2018_sicon}, the same authors investigate the value function associated to  a general class of infinite horizon optimal control problems that includes the control of Fokker-Planck equations.

Fleig and Guglielmi \cite{FleigGuglielmi2017} consider a distributed multidimensional control and adopt an homogeneous Dirichlet boundary condition. In that case, the measure $\int_\Omega \rho (t,x)dx $ is not preserved. For that problem, they show existence of optimal control  and first order optimality conditions. Similar results, for distributed control and for a non-flux boundary condition as ours, were obtained by Albi {\em et al.} \cite{AlbiChoiFornasierKalise2017}.

Equation \eqref{FP}, for a function $u$ that depends both on time and space, appears in second order Mean Field Games: see {\em e.g.} Gomes and Sa\'ude \cite[Theorem 2]{GomesSaude2014}, Lasry and Lions \cite[Section 2.6]{LasryLions2007}. 
\end{remark}

%%%%%%%%%%%%%%%%%%%%%%%%%%%%%%%%%%%%%%%%%%%%%%%%%%%%%%%%%%%%%%%%%%%%%%%%%%%%%%%%%%%%%%%%%%%%%%%%%%%%%%%%%%%%%%
\subsection{Existence and uniqueness of the solution of the Fokker-Planck equation} 
%%%%%%%%%%%%%%%%%%%%%%%%%%%%%%%%%%%%%%%%%%%%%%%%%%%%%%%%%%%%%%%%%%%%%%%%%%%%%%%%%%%%%%%%%%%%%%%%%%%%%%%%%%%%%%

%\deleted{We next make a formal derivation of the weak formulation of \eqref{FP}-\eqref{PBNeumann}.
%Consider $\varphi$ sufficiently smooth. From \eqref{FP} we get}
%\be
%\deleted{ \int_\Omega \big(\partial_t \rho  - \nu\Delta\rho - {\rm div}(\rho B[u]) \big) \varphi dx  = 0.}
%\ee
%\deleted{We have}
%\be
%\label{var2}
%\deleted{- \int_\Omega  \nu(\Delta\rho) \varphi 
% =  -\nu  \int_{\Gamma} \varphi \nabla \rho \cdot n ds  + \nu \int_\Omega \nabla \rho \cdot \nabla\varphi dx ,}
%\ee
%\be
%\label{var3}
%\deleted{- \int_\Omega{\rm div} (\rho B[u])\varphi dx 
% = - \int_\Gamma \varphi \rho B[u] \cdot n ds  +   \int_\Omega \rho B[u] \cdot \nabla \varphi dx .}
%\ee
%\deleted{Adding  \eqref{var2} and \eqref{var3}, in view of the boundary condition \eqref{PBNeumann} }
We start by obtaining the weak formulation of \eqref{FP}-\eqref{PBNeumann} by standard calculations. 
In particular, we apply the formula 
\[
- \int_\Omega{\rm div} (\rho B[u])\varphi dx 
 = - \int_\Gamma \varphi \rho B[u] \cdot n ds  +   \int_\Omega \rho B[u] \cdot \nabla \varphi dx 
\]
that will frequently be needed. Multiplying \eqref{FP} by a sufficiently smooth $\varphi$, integrating by parts,  and using the boundary condition \eqref{PBNeumann}, we get
\[
\intO \partial_t \rho\varphi dx + a[u(t)](\rho,\varphi) =0,
\]
where, for each $u\in \cR^n,$  $a[u](\cdot,\cdot)$ is a bilinear mapping that to each pair $\psi,\varphi \in H^1(\Omega)$ associates the value
\be
\label{adef}
a[u](\psi,\varphi):= \int_\Omega\big(\nu \nabla \psi + \psi B[u]\big) \cdot \nabla \varphi dx.
\ee
We will work in the space
\be 
\label{Wdef}
W(0,T) :=\left\{ \rho \in L^2(0,T;H^1(\Omega)) : \partial_t \rho\in L^2(0,T;H^{1}(\Omega)^*)\right\},
\ee
equipped with the norm
\[
\|\rho\|_{W(0,T)} := \left(\|\rho\|^2_{L^2(0,T;H^1(\Omega))} + \left\|\partial_t \rho\right\|_{L^2(0,T;H^1(\Omega)^*)}^2 \right)^{1/2}.
\]
It is known that $W(0,T)$ is a Hilbert space with the scalar product
\[
(\rho,\varphi)_{W(0,T)} := \intT(\rho(t),\varphi(t))_{\spaceV} dt +
\intT \big( \partial_t \rho(t) , \partial_t \varphi(t) \big)_{\spaceVp} dt,
\]
which induces the norm $\|\cdot\|_{W(0,T)}.$ We shall also recall the following continuous embedding (see {\em e.g.} Chipot \cite[Theorem 11.4]{Chipot2012} or Dautray-Lions \cite[Theorem 1, page 473]{DautrayLions1992})
$$
W(0,T) \hookrightarrow C \big( [0,T];L^2(\Omega) \big)
$$
that will be of use throughout this article.

The weak formulation of \eqref{FP} can be rewritten as to find $\rho \in W(0,T)$ such that
\begin{align}
\label{varfor1} \frac{d}{dt} (\rho (\cdot), \varphi) + a[u(\cdot)](\rho(\cdot),\varphi) &=0\quad \text{on } \mathcal{D}'(0,T)\text{ for all } \varphi \in H^1(\Omega),\\
\label{initcond} \rho(0) &= \rho_0\quad \text{in } \spaceH.
\end{align}

For convenience, let us consider a general right-hand side $f\in L^2(0,T;H^1(\Omega)^*)$ in \eqref{varfor1} and study the existence and uniqueness of the solution of the equation
 \be
\label{varforf} 
\frac{d}{dt} ( \rho(\cdot), \varphi) + a[u(\cdot)](\rho(\cdot),\varphi) =\langle f(\cdot),\varphi\rangle\quad \text{on } \mathcal{D}'(0,T)\text{ for all } \varphi \in H^1(\Omega),
\ee
with initial condition \eqref{initcond}.

\begin{remark}
Note that \eqref{varforf} can also be expressed as
\be
\partial_t \rho+ a[u](\rho,\cdot ) = f \quad \text{in } L^2(0,T;\spaceVp).
\ee
\end{remark}

\begin{definition}
We call a function $\rho \in W(0,T)$ a {\em weak solution} of the Fokker-Planck equation \eqref{FP}, if it satisfies \eqref{varfor1}-\eqref{initcond}. The definition is formulated analogously for \eqref{varforf}-\eqref{initcond}, where the r.h.s. of the differential equation in \eqref{FP} was replaced by an arbitrary element $f$ of the space $L^2(0,T;H^1(\Omega)^*).$ 
\end{definition}

\begin{remark}
For practical matters, we mention the equivalent variational formulation of \eqref{FP}-\eqref{PBNeumann} according to  Ladyzhenskaya \cite{Ladyzhenskaya_etal1968} that is as  follows: $\rho$ is a function in $W_2^{1,0}(Q)$ such that
\begin{multline}
\label{varfor}
\iint_Q \Big\{ -\rho \partial_t \varphi+  \big( \nu\nabla \rho + \rho B[u]\big)\cdot\nabla \varphi \Big\} dx\, dt =\int_\Omega \rho_0 \varphi(\cdot ,0) d x,\\ \text{for all }\varphi \in W^{1,1}_2(Q) \text{ with } \varphi(\cdot,T)=0.
\end{multline}
Here, $W_2^{1,0}(Q)$  is the Banach space of all $\rho \in L^2(Q)$ that have the weak derivatives
$\partial_{x_i} \rho$ in $L^2(Q)$ for all $i \in \{1,\ldots,n\}$ and $W_2^{1,1}(Q)$
is the subspace of all $\rho \in W_2^{1,0}(Q)$ that also possess the derivative $\partial_t \rho$ in $L^2(Q)$. We refer to \cite{Ladyzhenskaya_etal1968} for the definition and the associated norms.
\end{remark}

\if{ % DERIVATION OF THE WEAK FORMULATION
\begin{proof}
From \eqref{FP} we get 
\be
\int_0^T \int_\Omega \big(\partial_t \rho  - \nu\Delta\rho - {\rm div}(\rho B) \big) \varphi dx dt = 0.
\ee
We have
\be
\label{var1}
\int_0^T \int_\Omega \partial_t \rho\varphi = -\int_0^T \int_\Omega \rho \dot\varphi dxdt -\int_\Omega \rho_0(x) \varphi(x,0) dx,
\ee
\be
\label{var2}
\begin{split}
-\int_0^T \int_\Omega  \nu\Delta\rho \varphi &= -\nu \int_0^T \int_\Omega {\rm div} (\varphi \nabla \rho) dx dt + \nu\int_0^T \int_\Omega \nabla \rho \cdot \nabla\varphi dx dt\\
& =  -\nu \int_0^T \int_{\Gamma} \varphi \nabla \rho \cdot n ds ds + \nu\int_0^T \int_\Omega \nabla \rho \cdot \nabla\varphi dx dt,
\end{split}
\ee
\be
\label{var3}
\begin{split}
-\int_0^T \int_\Omega{\rm div} (\rho B)\varphi dx dt & = -\int_0^T \int_\Omega {\rm div}(\rho B\varphi)dx dt + \int_0^T \int_\Omega \rho B \cdot \nabla \varphi dx dt \\
& = -\int_0^T\int_\Omega \varphi \rho B \cdot n ds dt +  \int_0^T \int_\Omega \rho B \cdot \nabla \varphi dx dt.
\end{split}
\ee
Adding up \eqref{var1}, \eqref{var2} and \eqref{var3} yields the desired result.
\end{proof}
}\fi

% The problem of existence and uniqueness of a weak solution to \eqref{FP}-\eqref{PBNeumann}  was solved  by Breiten et al.  \cite{BreitenKunischPfeiffer2018} for controls $u \in L^2(0,T)$ and smooth domain. In the latter work, they also showed higher regularity of $\rho$ under appropriate assumptions. Their method can be extended to controls $u \in L^2(0,T;\mathbb{R}^n)$  and to a given inhomogeneous right-hand side $f \in L^2(0,T;H^1(\Omega)^*)$. For completeness and since we have assume less regularity on $\Omega$  later need the dependence of some continuity estimates on the control,  we include a proof for this extension but proceed in a slightly different way. We begin with bounded controls and extend this result to controls in $L^2(0,T;\mathbb{R}^n)$ by a density argument that is based on a priori estimates as in \cite{BreitenKunischPfeiffer2018}.

Existence and uniqueness of a weak solution to \eqref{FP}-\eqref{PBNeumann}  was proved  by Breiten {\em et al.}  \cite{BreitenKunischPfeiffer2018} for controls $u \in L^2(0,T)$. The result was also given by Albi {\em et al.} \cite{AlbiChoiFornasierKalise2017} for the more general control space $L^2(0,T;L^\infty(\Omega)).$ Both articles assume smoothness of the domain's boundary. Here we work on a Lipschitz domain and with 
controls in $\Utwo$. This extension to $\Utwo$ of the existence and uniqueness result for the weak solution of \eqref{FP}-\eqref{PBNeumann} is useful for the second order analysis we present later on.
Hence, we prove the result again, but in a different way. We begin with bounded controls and then extend the result to controls in $L^2(0,T;\mathbb{R}^n)$ by a density argument that is based on {\em a priori} estimates, as done in \cite{BreitenKunischPfeiffer2018}.

\begin{theorem}
\label{ThmExistSol}
Given $\rho_0\in L^2(\Omega),$ $u\in \Uinfty$, 
and $f\in L^2(0,T;H^{1}(\Omega)^*),$
there exists a unique weak solution $\rho$ of \eqref{varforf}-\eqref{initcond} 
and it belongs to the space $W(0,T)$.
\end{theorem}

\begin{proof}
We apply Lions-Magenes \cite[Theorem 4.1 and Remark 4.3, pp. 238-239]{LioMag68a}; see also Chipot \cite[Theorem 11.7]{Chipot2012}.

We first verify the assumptions of continuity and coercivity of the bilinear form $a[u]$ necessary to apply \cite[Theorem 4.1 and Remark 4.3, pages 238-9]{LioMag68a}. For $\rho,\varphi \in H^1(\Omega),$ we have
\be
\label{estarhophi}
\begin{split}
\big|a[u(t)](\rho,\varphi)\big| &\leq \left| \nu \int_\Omega \nabla \rho \cdot \nabla \varphi dx \right|  + \left| \int_\Omega  \rho B[u(t)]\cdot \nabla \varphi dx \right| \\
&\leq \nu\|\nabla \rho\|_{L^2(\Omega){^n}} \|\nabla \varphi\|_{L^2(\Omega){^n}} + \|B[u]\|_{L^\infty(Q){^n}} \|\rho\|_{L^2(\Omega)}\|\nabla \varphi\|_{L^2(\Omega)^n} \\
&\leq C\|\rho\|_{H^1(\Omega)} \|\varphi\|_{H^1(\Omega)},
\end{split}
\ee
where $C := \nu +  \|B[u]\|_{L^\infty(Q)^n}$. From  the latter estimate and Young's inequality, we also obtain, for any positive $M,$
\begin{equation*}
 \left| \int_\Omega  \rho B[u]\cdot \nabla \varphi ds \right|
 \leq \|B[u]\|_{L^\infty(Q)^n}  \left( \frac{M}{2}\|\rho\|_{L^2(\Omega)}^2 + \frac{1}{{2}M} \|\nabla\varphi\|^2_{L^2(\Omega)^n}\right).
 \end{equation*}
 Thus
\be
\label{acoercive}
\begin{split}
a[u(t)] & (\varphi,\varphi) 
= \nu  \int_\Omega |\nabla\varphi|^2 dx +  \int_\Omega  \varphi B[u(t)]\cdot \nabla \varphi dx\\
&\geq \nu \|\varphi\|^2_{H^1(\Omega)}  -\nu \|\varphi\|^2_{L^2(\Omega)}   -\|B[u]\|_{L^\infty(Q)^n}\left( \frac{M}{ 2}\| \varphi \|_{L^2(\Omega)}^2 + \frac{1}{{2}M} \|\nabla\varphi\|^2_{L^2(\Omega)^n}\right) \\
&\geq \gamma \|\varphi\|^2_{H^1(\Omega)} - \lambda \|\varphi\|_{L^2(\Omega)}^2,
\end{split}
\ee
if $M$ is chosen such that  $\displaystyle\gamma := \nu -\frac{\|B[u]\|_{L^\infty(Q)^n}}{2M}$ is positive, and we set $\lambda := \nu+ \ds \frac{M}{2}\, \|B[u]\|_{L^\infty(Q)^n.} $

We conclude from \eqref{estarhophi} and \eqref{acoercive} that the hypotheses of boundedness and coercivity of $a[u]$  in Lions-Magenes \cite[Theorem 4.1 and Remark 4.3, pp. 238-239]{LioMag68a} are satisfied and, therefore, their result can be applied. 
\end{proof}

Now we extend the existence result to $L^2$-controls. The proof via Gronwall's Lemma is inspired by that of \cite[Prop. 2.1]{BreitenKunischPfeiffer2018}.

\begin{theorem}\label{u2-case} 
For all $u \in L^2(0,T;\mathbb{R}^n)$, $f \in L^2(0,T;H^1(\Omega)^*)$,
and $\rho_0 \in L^2(\Omega)$,  the state equation \eqref{FP}-\eqref{PBNeumann} has a unique weak solution 
$\rho \in W(0,T)$. It obeys the estimate
\be
\label{W(0,T)}
\|\rho\|^2_{W(0,T)} \leq C_0 \Big(\|\rho_0\|^2_{L^2(\Omega)} +\|f\|_{L^2(0,T;\spaceVp)}^2\Big),
\ee
where the constant $C_0$ depends continuously on $\|u\|_2$ but is independent of $f$ and $\rho_0$.
\end{theorem}

\begin{proof} (i) First, we derive an {\em a priori} estimate for a solution $\rho$ of \eqref{FP}-\eqref{PBNeumann}. To this aim, we introduce a new function $\eta$ by $\rho(t) = e^{\nu t} \eta(t)$ and transform the Fokker-Planck equation \eqref{FP} into the equivalent equation
\[
\partial_t \eta - \nu \Delta \eta + \nu \eta- \div(\eta B[u]) = e^{-\nu t} f.
\]
Next, we test the weak formulation by  $\varphi = \eta$.  Select an arbitrary $t \in (0,T)$, and integrate over $(0,t)$ to get 
\begin{multline}
\int_0^t \frac{1}{2} \frac{d}{ds}\|\eta(s)\|_{L^2(\Omega)}^2ds + \nu \int_0^t \|\eta(s) \|_{H^1(\Omega)}^2 ds + \int_0^t \eta(s) B[u(s)] \cdot \nabla \eta(s) ds
\\= \int_0^t e^{-\nu s} \langle f(s),\eta(s)\rangle ds.  \label{Int-eq}
\end{multline} 
Notice that we have 
$
|B[u](x,t)| = |u(t)\otimes b(x) + c(x)| \le d_1 (|u(t)| + 1) \, \mbox{ a.e. in } Q,
$
since $b$ and $c$ are bounded. 
Applying Young's inequality  in a standard way to the third term on the l.h.s. and to the term in the r.h.s. of latter display,
%\tcr{ [the red terms can later be omitted] \[
%\begin{aligned}
%\frac{1}{2}\|\eta(t)\|^2 +  \nu \int_0^t \|\eta(s) \|_{H^1(\Omega)}^2 ds \le & \frac{1}{\nu} \int_0^t \|\eta(s)\|_{L^2(\Omega)}^2c_1\sol{^2} \, \big(|u(s)| + 1\big)^2ds + \frac{\nu}{4}\int_0^t |\nabla\eta(s)|^2ds \\&\quad +\frac{1}{2}\|\rho_0\|_{L^2(\Omega)}^2+ \frac{1}{\nu} \int_0^t \|f(s)\|_{H^{1}(\Omega)^*}^2ds + \frac{\nu}{4}\int_0^t \|\eta(s)\|^2_{H^1(\Omega)}ds
%\end{aligned} \] }
and compensating the $H^1$-terms on $\eta$ on both sides of the inequality, we arrive at
\begin{multline} \label{aux*}
\frac{1}{2}\|\eta(t)\|_{L^2(\Omega)}^2+ \frac{\nu}{2} \int_0^t \|\eta(s)\|^2_{H^{1}(\Omega)}ds\\ \le 
\frac{1}{\nu}\|f\|_{L^2(0,T;H^1(\Omega)^*)}^2
+\frac{1}{2} \|\rho_0 \|_{L^2(\Omega)}^2 
+ \int_0^t \|\eta(s)\|_{L^2(\Omega)}^2\frac{d_1^2}{\nu}(|u(s)|+1)^2ds.
\end{multline}
%hence 
%\[
%\|\eta(t)\|_{L^2(\Omega)}^2\le 
%\frac{2}{\nu}\|f\|_{L^2(0,T;H^1(\Omega)^*)}^2
%+ \|\rho_0 \|_{L^2(\Omega)}^2
%+ \int_0^t \|\eta(s)\|_{L^2(\Omega)}^2\frac{2c_1}{\nu}(|u(s)|+1)^2ds.
%\]
By Gronwall's inequality for  $z(t) := \|\eta(t)\|_{L^2(\Omega)}^2$ and by the estimate  $(|u|+1)^2 \le 2 (|u|^2 + 1)$, the inequality
\[
\|\eta(t)\|_{L^2(\Omega)}^2 \le \left(\frac{2}{\nu}\|f\|_{L^2(0,T;H^1(\Omega)^*)}^2
+ \|\rho_0 \|_{L^2(\Omega)}^2\right) \exp\left(\int_0^t \frac{4d_1^2}{\nu} (|u(s)|^2+1)ds\right)
\]
is deduced. Therefore, since $t \in (0,T)$ was taken arbitrarily, 
\[
\|\eta\|_{L^\infty(0,T;L^2(\Omega))}^2 \le d_2 \left(\|f\|_{L^2(0,T;H^1(\Omega)^*)}^2
+ \|\rho_0 \|_{L^2(\Omega)}^2\right)
\]
follows in turn, where $d_2$ depends continuously on $\|u\|_2$. Inserting this estimate in \eqref{aux*}, an analogous inequality for 
$\|\eta\|_{L^2(0,T;H^1(\Omega))}$ is derived that implies an estimate for  $\|\partial_t \eta\|_{L^2(0,T;H^1(\Omega)^*)}$ in a standard way.  
\if{\fredi{In particular, here we use that the functional
\begin{equation}\label{reg_eta}
\varphi \mapsto \iint_Q \eta B[u]\cdot \nabla \varphi \, dxdt
\end{equation} 
is continuous in $L^2(0,T;H^1(\Omega))$, if $u \in \Uinfty$ and $\eta \in L^2(0,T;H^1(\Omega))$.
}}\fi
Altogether, this finally permits to deduce that
\be\label{etarho}
\|\eta\|_{W(0,T)} \le c_3 \left(\|f\|_{L^2(0,T;H^1(\Omega)^*)}^2
+ \|\rho_0 \|_{L^2(\Omega)}^2\right),
\ee
where $c_3$ depends continuously on $\|u\|_2$. Transforming back by $\rho = e^{\nu t} \eta$ finally yields
 \eqref{W(0,T)} as {\em a priori} estimate.

(ii) Now we fix an arbitrary $u \in \Utwo$ and show that \eqref{FP}-\eqref{PBNeumann} has a solution.
To this end, we select a sequence $\{u_k\} \subset \Uinfty$ such that $\|u_k - u\|_2 \to 0$ for $ k \to \infty$. We can assume that $\|u_k\|_2 \le 2 \, \|u\|_2,$ for all $k \in \mathbb{N}$. 

Thanks to Theorem \ref{ThmExistSol}, to each $u_k,$ a unique solution $\rho_k \in W(0,T)$ of \eqref{FP}-\eqref{PBNeumann} exists. By the {\em a priori} estimate \eqref{etarho} obtained in (i) above, the sequence $\{\rho_k\}$ is bounded in 
$W(0,T)$ and we can select a subsequence that converges weakly to some $\rho \in W(0,T)$. Possibly after renumbering, we can assume that $\rho_k \rightharpoonup \rho$. 

Passing to the limit in the weak formulation \eqref{varfor}, it is easy to confirm that $\rho$ solves \eqref{FP}-\eqref{PBNeumann}. Here, the term $\rho_k u_k \otimes b \cdot \nabla \varphi$ needs some special care. Let us confirm that $\rho_k u_k \otimes b$ converges weakly to $\rho u \otimes b$ in $L^2(Q;\cR^n)$. Indeed, for all $v \in L^2(Q;\cR^n)$,
\[
\iint_Q \rho_k (u_k \otimes b) \cdot v\, dxdt 
= \int_0^T u_k \cdot \left(\int_\Omega \rho_k b \otimes v \,dx\right)dt
\underset{k \to \infty}{\to} \int_0^T u \cdot \left(\int_\Omega \rho b \otimes v \,dx\right)dt,
\]
since $\rho_k \rightharpoonup \rho$ in $L^\infty(0,T;L^2(\Omega))$ follows from the weak convergence of $\{ \rho_k \}$ to $\rho$ in 
$W(0,T)$, from the continuous embedding of this space in $C([0,T],L^2(\Omega))$, and from the fact that
the mapping $\rho \mapsto \int_\Omega \rho b \otimes v\, dx$ is linear and continuous from $L^\infty(0,T;L^2(\Omega))$ to $L^2(0,T;\mathbb{R}^n)$. The uniqueness follows in a standard way by the {\em a priori} estimates.
\end{proof}

\if{
\begin{proposition}
Given $\rho_0,u,f$ as in Theorem \ref{ThmExistSol}, the unique weak solution $\rho$ of \eqref{FP}-\eqref{PBNeumann} satisfies
\be
\label{intrho1}
\intO \rho(x,t) dx = \intO \rho_0(x) ds\quad \text{for a.a. } t\in (0,T).
\ee
\end{proposition}

\begin{proof}
This known fact follows from \eqref{varfor1} by choosing $\varphi \equiv 1.$ 

\if{If not, the following proof holds.

To obtain the same result from \eqref{varfor}, we can do as follows. Fix $s\in (0,T).$ Consider a test function $\varphi:=\varphi(t)$, with $\varphi \in W^{1,1}_2(Q)$ and such that
\be
\varphi(t) = 
\left\{
\begin{split}
1,\quad &t\in [0,s],\\
-\frac{1}{\epsilon}(t-\epsilon-s),\quad & t\in (s,s+\epsilon],\\
0,\quad & t\in (s+\epsilon,T].
\end{split}
\right.
\ee
Then, the variational formulation \eqref{varfor} gives
\be
-\intT\intO \rho(x,t) \dot\varphi(t) dx dt = \intO \varphi_0(x)\varphi(0)dx,
\ee
which implies
\be
\int_s^{s+\epsilon} \frac{1}{\epsilon} \intO \rho(x,t) dx dt = \intO \varphi_0(x) dx.
\ee
Setting $R(t):= \intO \rho(x,t) dx,$ we can rewrite latter display as
\be
\frac{1}{\epsilon} \int_s^{s+\epsilon} R(t) dt = \intO \varphi_0(x) dx.
\ee
For a.e. $s\in (0,T)$, the l.h.s. of latter equation converges to $R(s),$ thus we get \eqref{intrho1} as desired.
}\fi
\end{proof}
}\fi

\begin{remark}[On the mass conservation and nonnegativity of the solutions]
Note that, by choosing $\varphi \equiv 1$ in \eqref{varfor1}, it follows that 
\[
\intO \rho(x,t) dx = \intO \rho_0(x) ds\quad \text{for a.a. } t\in (0,T),
\]
where $\rho_0$ is the initial condition and $\rho$ the corresponding weak solution of \eqref{FP}-\eqref{PBNeumann}.
Moreover, from  Chipot \cite[Theorem 11.9, p. 202]{Chipot2012} a Weak Maximum Principle for the weak solutions of \eqref{FP}-\eqref{PBNeumann} follows, from which we can deduce, in particular, that weak solutions are nonnegative whenever the initial condition is nonnegative.
\end{remark}

\if{
\begin{proposition}[Weak maximum principle]
\label{Th11.9-Chipot}
Fix $u\in \Uinfty,$ $\rho_0^1,\rho_0^2 \in L^2(\Omega)$ and $f^1,f^2 \in L^2(0,T;\spaceVp).$ Let $\rho_1,\rho_2$ be the weak solutions of \eqref{FP}-\eqref{PBNeumann} with control $u,$  associated to $\rho_0^1, f^1$ and $\rho_0^2,f^2,$ respectively. Suppose further that 
\[
\rho_0^1 \leq \rho_0^2\quad \text{and} \quad f^1 \leq f^2,
\]
where the second inequality means $\langle f_1,\varphi \rangle \leq \langle f_2,\varphi \rangle$ for every $\varphi \in L^2(0,T;H^1(\Omega))$ such that  $\varphi \geq 0$ a.e. in $Q.$
Then we have that
\[
\rho_1\leq \rho_2\quad \text{a.e. in }Q.
\]
\end{proposition}

\begin{proof} This result follows straightforwardly from Chipot \cite[Theorem 11.9, p. 202]{Chipot2012}.

% \tcb{Please keep the hidden proof below for the case that the referees want to have it.} 
% OK!

\if{This proof is taken from Chipot {\cite[Theorem 11.9, p. 202]{Chipot2012}}.
We have
\be
\frac{d}{dt} \big( (\rho_1-\rho_2)(t),\varphi \big) + a[u](t;(\rho_1-\rho_2)(t),\varphi) = \left\langle \tcr{(f^1-f^2)(t)},\varphi \right\rangle,\quad 
\ee
\tcr{for all  $\varphi \in H^1(\Omega)$.}
We know, by Riesz Representation Theorem, that we can write 
\be
a[u](t;\rho,\cdot) = (A(t)\rho)(\cdot),
\ee
\tcr{where $A(t)$ is ...}

We get
\be
\frac{d}{dt} (\rho_1-\rho_2)\tcr{(t)} +A(t)(\rho_1-\rho_2)\tcr{(t)} =  \tcr{(f^1-f^2)(t)},\quad \text{in } L^2(0,T;\spaceVp).
\ee
In view of \cite[Theorem 2.8, p28]{Chipot2012} we have that $(\rho_1-\rho_2)^+ \in L^2(0,T;H^1(\Omega)),$ then
\be
\langle (\rho_1-\rho_2)_t,(\rho_1-\rho_2)^+ \rangle + a[u](\deleted{t}\tcr{\cdot};\rho_1-\rho_2, (\rho_1-\rho_2)^+ ) = \langle f^1-f^2,(\rho_1-\rho_2)^+ \rangle.
\ee
By \cite[Lemma 11.2]{Chipot2012},  $\displaystyle\frac{d}{dt} \|(\rho_1-\rho_2)^+\|^2_{L^2(\Omega)} = \frac{d}{dt} \langle \rho_1-\rho_2,(\rho_1-\rho_2)^+ \rangle,$ thus
\be
\label{ddtrhoi}
\frac{d}{dt} \|(\rho_1-\rho_2)^+\|^2_{L^2(\Omega)}  + a[u](\deleted{t} \tcr{\cdot};(\rho_1-\rho_2)^+, (\rho_1-\rho_2)^+ ) = \langle f^1-f^2,(\rho_1-\rho_2)^+ \rangle,
\ee
where we used $a[u](t;\rho_1-\rho_2, (\rho_1-\rho_2)^+ ) =a[u](t;(\rho_1-\rho_2)^+, (\rho_1-\rho_2)^+ )$ ({\bf XXX: Justify this}).
Using the property \eqref{acoercive} and $f^1\leq f^2,$ we obtain from \eqref{ddtrhoi} that
\be
\frac{d}{dt} \|(\rho_1-\rho_2)^+\|^2_{L^2(\Omega)}  - \lambda \|(\rho_1-\rho_2)^+\|^2_{L^2(\Omega)} \leq 0.
\ee
Setting $z(t) : = \|(\rho_1-\rho_2)^+(t)\|^2_{L^2(\Omega)},$ we have
$z'(t) - \lambda z(t) \leq 0, z(0)=0.$ Thus $z(t) \leq 0$ for a.e. $t\in [0,T].$ 
We obtain that $\|(\rho_1-\rho_2)^+(t)\|^2_{L^2(\Omega)} = 0$ for a.e. $t\in [0,T].$ This concludes the proof.
}\fi
\end{proof}

\begin{corollary}
The weak  solutions $\rho$ of \eqref{FP}-\eqref{PBNeumann} are nonnegative whenever the initial condition $\rho_0$ is nonnegative.
\end{corollary}

\begin{proof}
The claim follows from Proposition \ref{Th11.9-Chipot} by setting $\rho_0^1\equiv0{;}$
see also Breiten {\em et al.} \cite{BreitenKunischPfeiffer2018}.
\end{proof}
}\fi

%%%%%%%%%%%%%%%%%%%%%%%%%%%%%%%%%%%%%%%%%%%%%%%%%%%%%%%
\section{The optimal control problem} \label{S3}
\setcounter{equation}{0}
%%%%%%%%%%%%%%%%%%%%%%%%%%%%%%%%%%%%%%%%%%%%%%%%%%%%%%%

In our optimal control problem, we minimize the cost functional
\be
\label{cost}
\begin{split}
J(\rho,u) :=&\frac{\alpha_Q}{2} \intT \|\rho(t)-\rho_Q(t)\|^2_{L^2(\Omega)} dt + \frac{\alpha_{\Omega}}{2}\|\rho(T)-\rho_\Omega\|_{L^2(\Omega)}^2 \\
&+ \sum_{i=1}^n \left(\frac{\gamma_i}{2}\|u_i\|^2_{2} + \beta_i \intT u_i(t) dt \right),
\end{split}
\ee
with {$\rho_Q \in L^2(Q), \rho_\Omega \in L^2(\Omega),$ and} $\beta_i,\gamma_i\geq 0$ for $i=1,\dots,n,$
subject to  the control constraints
\be
\label{uconstraint}
{u^{\min}}(t) \leq u(t) \leq u^{\max}(t)\qquad \text{for a.e. } t\in [0,T],
\ee
where the inequalities are defined componentwise, and the bounds $u^{\min},u^{\max}$ belong to $L^\infty(0,T;\cR^n).$ The set of admissible controls is
\be \label{uad}
\mathcal{U}_{\rm ad} : = \{u\in L^\infty(0,T;\cR^n): \, \eqref{uconstraint}\, \text{holds }\}.
\ee
The parameters $\beta_i,\gamma_i$, $i=1,\dots,n,$ are allowed to vanish simultaneously unless
second order sufficient optimality conditions are investigated. Then, all $\gamma_i$ have to be positive.
\begin{definition}\label{DefG}
Let us define the {\em control-to-state mapping}
\[
G\colon \Utwo \to W(0,T),
\]
that associates to each $u \in \Utwo$ the unique weak solution $\rho \in W(0,T)$ of \eqref{FP}-\eqref{PBNeumann}.
When necessary, we may write $G(u)$ to denote the state $\rho$ corresponding to $u$.
\end{definition}
Our optimal control problem can be rewritten as
\be
\label{P}\tag{P}
\min_{u \in \mathcal{U}_{\rm ad}} \,\, J(G(u),u).
\ee

We study two types of solutions, that we define next. We say that $\ub\in \Uad$ is an {\em $L^\infty$-local solution}  (resp., {\em $L^2$-local solution}) of \eqref{P} if there exists $\eps>0$ such that $J(G(\ub),\ub)\leq J(G(u),u)$ holds for every $u \in \Uad \cap B^\infty_\eps(\ub)$ (resp., $u \in \Uad \cap B^2_\eps(\ub)$), where $B^\infty_\eps(\ub)$ denotes the open ball in $\Uinfty$ (resp., in $\Utwo$) of radius $\eps$ centered at $\bar u$.

%%%%%%%%%%%%%%%%%%%%%%%%%%%%
\subsection{Existence of optimal controls}
%%%%%%%%%%%%%%%%%%%%%%%%%%%%

\begin{theorem}[Existence of an optimal control]
\label{ThmExist}
There exists (at least) one optimal control for \eqref{P}.
\end{theorem}
\begin{proof}
In view of the control constraints \eqref{uconstraint}, the set of admissible controls $\Uad$ is bounded in $L^\infty(0,T;\cR^n).$  Hence, thanks to the estimate \eqref{W(0,T)}, the set of states associated to admissible controls is bounded in $W(0,T).$ Therefore, the cost functional $J$ is bounded from below on the set of admissible state-control pairs. Thus, there exists a minimizing sequence $\{(\rho_k,u_k)\} \subset W(0,T) \times \Uad,$ where $\rho_k:=G(u_k)$, such that 
\[
J(\rho_k,u_k ) \longrightarrow \inf_{u\in \Uad} J(G(u) , u).
\]
Since $\{u_k\}$ is bounded in $L^\infty(0,T;\cR^n),$ it contains a weakly$^*$  converging subsequence, thus, keeping the same index, we have
\[
u_k \overset{*}{\rightharpoonup} \bar{u}\quad \text{in } L^\infty(0,T;\cR^n)
\]
for some $\bar{u} \in L^\infty(0,T;\cR^n).$
The corresponding sequence of states $\{\rho_k\}$ forms a bounded sequence in $W(0,T).$ Thus, there exists $\bar\rho\in W(0,T)$ such that (extracting if necessary a subsequence)
\be
\label{limrhok}
\rho_k \rightharpoonup \bar\rho\quad \text{in } W(0,T).
\ee
The objective functional $J$ is weakly lower semicontinuous, so we obtain that 
\[
J(\rhob, \ub) \le \liminf_{k \to \infty} J(\bar \rho_k, \bar u_k)
\]
and hence, $\bar u$ is optimal provided that $\rhob$ is the associated state.

The main work of the proof is to show that $\rhob$ is the state associated to $\ub,$ and then $J(\rhob,\ub) = \inf_{u\in \Uad} J(G(u) , u).$ For this, we prove that we can pass to the limit in \eqref{varfor}.

Using Aubin-Lions' Lemma \cite{Aubin1963}, we can deduce from \eqref{limrhok} that $\{\rho_k\}$ has a subsequence converging  to $\rhob$ strongly in $L^2(Q)$. This is, keeping the same index for the subsequence,
\[
\rho_k \longrightarrow \bar\rho\quad \text{(strongly) in } L^2(Q).
\]
Putting all together, we have that $\rho_k \rightharpoonup \rhob$ weakly in $L^2(0,T;\spaceV)$,  $\ds \partial_t \rho_k \rightharpoonup \partial_t \rhob$ weakly in $L^2(0,T;\spaceVp),$ and $\rho_k\to\rhob$ strongly in $L^2(Q)$.

We next show that $\bar\rho = G(\ub).$

Given $\varphi \in W^{1,1}_2(Q),$ with $\varphi(\cdot,T)=0,$ we have
\begin{multline*}
\intT\left(\left\langle\partial_t \rho_k (t), \varphi(t) \right\rangle + \intO \nu\nabla\rho_k(t)\cdot \nabla\varphi(t)dx\right)dt \to \\
\intT\left( \left\langle \partial_t \bar\rho (t), \varphi(t) \right\rangle + \intO \nu\nabla\rhob(t)\cdot \nabla\varphi(t)dx \right)dt.
\end{multline*}
In the weak formulation of the state equation, it only remains to check the convergence of the part containing $B[u_k].$ To this aim,
for $i=1,\dots,n,$ we consider the terms 
\be
\label{termurho}
\rho_k  B_i[u_k] \frac{\partial \varphi}{\partial x_i} 
= \rho_k \big(c_i +  b_i u_{k,i} \big) \frac{\partial \varphi}{\partial x_i}.
\ee
It is easy to see that
$\ds \rho_k c_i \frac{\partial \varphi}{\partial x_i} \to \rhob c_i \frac{\partial \varphi}{\partial x_i}$ and  $\ds \rho_k b_i \frac{\partial \varphi}{\partial x_i} \to \rhob b_i \frac{\partial \varphi}{\partial x_i}$ strongly in $L^1(Q)$, 
since $\rho_k \to \rhob$ strongly in $L^2(Q)$,  $c_i,b_i\in L^\infty(\Omega)$ and $\ds\frac{\partial \varphi}{\partial x_i} \in L^2(Q).$ 
Therefore, in view of the weak$^*$ convergence of $\{u_k\}$ to $\ub$, we get that
\[
\iint_Q u_k  \rho_k b_i \frac{\partial \varphi}{\partial x_i} dx\, dt \to 
\iint_Q \ub \rhob b_i \frac{\partial \varphi}{\partial x_i} dx \, dt, \quad \text{for } k \to \infty.
\]
 Thus, $\rhob$ is the state associated to $\ub.$  This concludes the proof.
\end{proof}

We are not able to prove uniqueness of optimal controls, since the control-to-state mapping is nonlinear. 
Therefore, the reduced objective functional is nonconvex in general.

%%%%%%%%%%%%%%%%%%%%%%%%%%%%%%%%
\section{Properties of the control-to-state mapping} \label{S4}
\setcounter{equation}{0}
%%%%%%%%%%%%%%%%%%%%%%%%%%%%%%%%
In this section, we prove Fr\'echet differentiability of the control-to-state mapping $G$  and derive  the first order necessary optimality condition of Proposition \ref{PropFirst} as a corollary.

%%%%%%%%%%%%%%%%%%%%%%%%%%%%%%%%%%%%
\subsection{Fr\'echet differentiability of the control-to-state mapping}
%%%%%%%%%%%%%%%%%%%%%%%%%%%%%%%%%%%%
We start by proving the differentiability of $G$ by the Implicit Function Theorem.
To this aim, we define the mapping 
\be
\label{calG}
(\rho,u) \longmapsto  \mathcal{G}(\rho,u):= \left(\partial_t \rho+a[u](\rho,\cdot),\rho(0)-\rho_0\right)
\ee
from $W(0,T)\times \Utwo$ to $L^2(0,T;\spaceVp) \times \spaceH.$ The state equation can be viewed as the equation
\be
\mathcal{G}(\rho,u) =0.
\ee

\begin{proposition}
\label{GCinfty}
The mapping $\mathcal{G}$ is of class $C^\infty.$
\end{proposition}

\begin{proof}
The first component $\mathcal{G}_1$ of $\mathcal{G}$ is defined by
\[
\mathcal{G}_1(\rho,u)(\varphi) := \left\langle \partial_t \rho , \varphi\right\rangle + \int_\Omega \nabla \rho \cdot \nabla \varphi \, dxdt + \int_\Omega \rho B[u] \cdot \nabla \varphi \, dxdt.
\]
Its first two summands clearly define linear and continuous mappings from $W(0,T)$ to $L^2(0,T;H^1(\Omega)^*)$, hence they are of class $C^\infty$. Therefore, it suffices to confirm that the operator
\[
(\rho,u) \mapsto \rho B[u]
\]
is of class $C^\infty$ from  $W(0,T)\times \Utwo$ to $L^2(Q;\cR^n)$. This, however, follows easily from the quadratic nature of $\rho B[u]$. For increments $\sigma \in W(0,T)$ and $v \in \Utwo$, we have
\[
\begin{aligned}
(\rho + \sigma)B[u + v] &= (\rho + \sigma)\big( c + b \otimes(u+v)\big) = \rho B[u] + \rho b \otimes v + \sigma(c + b \otimes u) + \sigma b \otimes v.
\end{aligned}
\]
This is a second order Taylor expansion of the mapping $(\rho,u) \mapsto \rho B[u]$ with continuous linear and quadratic parts, and vanishing remainder term. Let us exemplarily confirm the continuity of the quadratic term. We have 
\[
\|\sigma b \otimes v\|_{L^2(Q)^n} \le \| b \|_{L^\infty(\Omega)^n} \|\sigma\|_{C([0,T];L^2(\Omega))}\|v\|_2
\le c \|\sigma\|_{W(0,T)}\|v\|_2,
\]
hence the continuity of the quadratic form $\sigma b \otimes v$. Therefore $\mathcal{G}_1$ is of class 
$C^\infty$. The second component of $\mathcal{G}$ is obviously of class $C^\infty$.
\end{proof}

We thank the anonymous referee, who pointed out that our first proof in \cite{AronnaTroeltzsch2020v1} already covered the differentiability of $\mathcal{G}$ with respect to $u$ in $\Utwo$. This paved the way for proving the differentiability of $G$ in $\Utwo$ rather than in $\Uinfty$ as in \cite{AronnaTroeltzsch2020v1}.

\begin{corollary}
\label{CorGCinfty}
The control-to-state mapping $G\colon \Utwo \to W(0,T)$ is of class $C^\infty.$
\end{corollary}

\begin{proof}
In view of Proposition \ref{GCinfty}, we have that $\mathcal{G}$ is of class $C^\infty$ and that
\[
\partial_\rho \mathcal{G}(\rho,u)z = \big( \partial_t z+a[u](z,\cdot),z(0)\big).
\]
Given $u\in \Utwo$, $f\in L^2(0,T;\spaceVp)$ and $z_0\in \spaceH,$ the system
\be
\begin{split}
\partial_t z+a[u](z,\cdot) &= f,\\
z(0) &= z_0,
\end{split}
\ee
is again a Fokker-Planck equation as the state equation and, therefore, it has a unique (weak) solution $z[z_0,f]$  that belongs to $W(0,T),$ and depends continuously on $z_0 \in L^2(\Omega)$ and on $f \in L^2(0,T;\spaceVp).$ 
Therefore, thanks to the existence and uniqueness Theorem \ref{u2-case}, $\ds \partial_\rho \mathcal{G}(\rho,u)$ is an isomorphism from 
$W(0,T)$ to $L^2(0,T;\spaceVp) \times \spaceH$.
Thus, the hypotheses of the Implicit Function Theorem are satisfied, and then $\mathcal{G}(\rho,u)=0$ implicitly defines the control-to-state operator $G: u \mapsto \rho$ that is itself of class $C^\infty.$
\end{proof}

In the next proposition and for other occasions throughout the remainder of the article, we will use the following type of functional: for  given $w \in L^2(\Omega;\cR^n)$, we introduce $d[w] \in H^1(\Omega)^*$ defined by
\[
d[w](\varphi) := - \int_\Omega w \cdot \nabla \varphi\, dx
\]
for all $\varphi\in \spaceV.$

\begin{proposition}
\label{PropG'}
Let $u\in \Utwo$ and $\rho:=G(u).$
For any $v\in \Utwo,$ we have that $G'(u)v=z,$ where $z$ belongs to $W(0,T)$ and is the weak solution of the {\em linearized state equation} at $(\rho,u)$ that is given by
\be
\label{lineq}
\begin{split}
\partial_t z+a[u](z,\cdot) &= d[\rho b\otimes v],\\
z(0) &= 0.
\end{split}
\ee
Moreover, the following estimate holds,
\be
\label{estimatez}
\|z\|_{{C([0,T];\spaceH)}} + \|z\|_{W(0,T)} 
\leq K \| b\|_{L^\infty(\Omega)^n} \|\rho\|_{{C([0,T];\spaceH)}} \|v\|_2, 
\ee
where $K$ depends continuously on $\|u\|_{2}$ but does not depend on $v$.
\end{proposition}

\begin{proof}
The fact that \eqref{lineq} possesses a unique weak solution $z$ in $W(0,T)$ follows from Theorem \ref{u2-case} as already observed in the proof of Corollary \ref{CorGCinfty} above.
The representation $G'(u)v=z$ follows from a direct application of the Implicit Function Theorem by differentiating
the state equation \eqref{FP}-\eqref{PBNeumann} with respect to the control.

It remains to estimate the $L^2(0,T;\spaceVp)$-norm of the r.h.s. function $d[\rho b\otimes v]$, in order to apply the estimates of Theorem \ref{u2-case}. Take any $\varphi \in \spaceV;$ then for a.a. $t \in [0,T]$ we have
\begin{equation*}
\begin{split}
\big| d[\rho(t) b\otimes v(t)](\varphi)\big| &= 
\left| \intO \rho(t) ( b\otimes v(t)) \cdot \nabla\varphi dx \right| \\
%&\leq \sum_{i=1}^n \left| v_i(t) \intO \rho(x,t) b_i(x)\frac{\partial\varphi}{\partial x_i}(x) dx\right|\\
&\leq \sum_{i=1}^n \left| v_i(t)\right|  \| b_i\|_{L^\infty(\Omega)} \|\rho(t)\|_{L^2(\Omega)}  \|\varphi\|_{\spaceV}.
\end{split}
\end{equation*}
Thus, for a.e. $t\in [0,T],$
$$
\left\| d[\rho(t) b\otimes v(t)] \right\|_{\spaceVp}\leq \| b\|_{L^\infty(\Omega)^n} \|\rho(t)\|_{\spaceH} \sum_{i=1}^n \left| v_i(t)\right|,
$$
and hence
$$
\left\|d[\rho b\otimes v]\right\|_{L^2(0,T;\spaceVp)}\leq M \| b\|_{L^\infty(\Omega)^n} \|\rho\|_{L^2(0,T;\spaceH)} \|v\|_2,
$$
for some constant $M$ depending on the control dimension $n$.
The estimate \eqref{estimatez} and the continuous dependence of $K$ on $\|u\|_2$ follow from Theorem \ref{u2-case} with $f:= d[\rho b\otimes v]$ and $\rho_0 = 0$. This concludes the proof.
\end{proof}

The linearized state equation \eqref{lineq} at $(\rho,u)$ in the direction of $v$ is to be understood as
\be
\label{lineareq2}
\begin{split}
\intO \big( \partial_t z\varphi + (\nu\nabla z+zB[u])\cdot \nabla \varphi \big) dx &= - \intO \rho ( b\otimes v)\cdot \nabla \varphi dx \,\, \text{ for all } \varphi \in \spaceV,\\
z(0) &= 0,
\end{split}
\ee
or, in the strong form:
\be
\label{lineareq3}
\begin{split}
\partial_t z-\nu\Delta z-{\rm div} (zB[u]) &= \div \big( \rho (b \otimes v) \big)\quad \text{in } Q,\\
\big(\nu\nabla z+zB[u]\big)\cdot n &= - \big(\rho  b\otimes v\big)\cdot n\quad \text{in } \Sigma,\\
z(0) &= 0\quad \text{in } \Omega.
\end{split}
\ee

%\deleted{In fact, integration by parts in the last two terms in the l.h.s. and in the term in the r.h.s of \eqref{lineareq2} yields, respectively,}
%\begin{gather*}
%\deleted{\intO \nu\nabla z \cdot \nabla \varphi dx = \nu \int_\Gamma \varphi\nabla z \cdot n ds - \nu \intO\varphi\Delta z dx,}\\
%\deleted{\intO zB[u] \cdot \nabla \varphi dx = \int_\Gamma \varphi z B[u]\cdot nds - \intO\varphi {\rm div} (zB[u]) dx,}\\
% \deleted{-\intO \rho ( b\otimes v)\cdot \nabla \varphi dx = -\int_\Gamma\varphi\rho ( b\otimes v)\cdot n ds + \intO  \varphi {\rm div}(\rho( b\otimes v)) dx },
%\end{gather*}
%\deleted{from which \eqref{lineareq3} follows.}
%

\begin{remark}
Following Ladyzhenskaya {\em et al.} \cite{Ladyzhenskaya_etal1968},  an equivalent variational formulation of the linearized state equation \eqref{lineareq2} is given as follows: $z$ is a function in $W_2^{1,0}(Q)$ such that
\begin{multline}
\label{lineareqLady}
\iint_Q \left( - z \partial_t \varphi + (\nu\nabla z+zB[u])\cdot \nabla \varphi \right) dx\, dt = - \iint_Q \rho (b\otimes v) \cdot \nabla \varphi dx\,dt\\
\text{for all }\varphi \in W^{1,1}_2(Q) \text{ with } \varphi(\cdot,T)=0.
\end{multline}
The additional property that $z \in W(0,T)$ is a standard consequence of \eqref{lineareqLady}.
\end{remark}

%\tcr{I have shifted this Proposition (formerly Prop. 5.3) to this point in order to simplify the proof of Lipschitz continuity of $G$. This follows partially a suggestion of the referee. But only partially. You will see it below.}

Let us define the {\em reduced cost functional} as
\[
F(u):=J(G(u),u).
\]
By the chain rule, the reduced cost functional $F$ is continuously Fr\'echet differentiable, since $J$ and $G$ have this property.

\begin{proposition}[First order necessary condition]
\label{PropFirst}
If $\ub$ is an $L^\infty$-local minimum for \eqref{P}, then
\be
\label{FirstAbstract}
F'(\ub)(u-\ub) \geq 0\quad \text{for every } u\in \mathcal{U}_{\rm ad}.
\ee
\end{proposition}

\begin{proof}
The proof is standard and follows straightforwardly by writing the Newton quotient of $F$.
\end{proof}

The variational inequality \eqref{FirstAbstract} also holds for $L^2$-local minima, because
any $L^2$-local minimum is also an $L^\infty$-local one. 

%%%%%%%%%%%%%%%%%%%%%%%%%%%%%%%%
\subsection{Lipschitz continuity of the control-to-state mapping}
%%%%%%%%%%%%%%%%%%%%%%%%%%%%%%%%

We conclude this section by proving the local Lipschitz continuity of $G$ that mainly follows from its
differentiability. However, we are particularly interested in the continuous dependence of the Lipschitz constant on the control.

\begin{proposition}\label{Propdeltarho}The control-to-state mapping $G$ is locally Lipschitz, i.e. 
for any pair $u_1,u_2 \in L^{2}(0,T;\cR^n)$ with associated states $\rho_1:=G(u_1)$, $\rho_2:=G(u_2)$, one has
\be
\label{estdeltarho}
 \|\rho_2-\rho_1 \|^2_{W(0,T)} \leq C\|\rho_1\|_{C([0,T];\spaceH)}^2\|u_2-u_1\|_2^2,
\ee
where  $C$ depends continuously on $\|u_2\|_{2}$.
\end{proposition}

%\tcr{This was our old proof. We could do it now as follows: The first equation in this proof is the linearized  equation. We set $\delta \rho = z$ and $v = u_2-u_1$. And now we instantly get the estimate via Proposition \ref{Propestimatez}. We do not need the mean value theorem in integral form as suggested by the referee!} 

%
\begin{proof}  
%% Shortened version! - 11th November %%
Let us consider two weak solutions $\rho_1, \rho_2$ of \eqref{FP}-\eqref{PBNeumann}, associated to $u_1,u_2$, respectively.
Setting $\delta\rho:=\rho_2-\rho_1,$  we have, for any $\varphi\in \spaceV$,
\[
\frac{d}{dt} \intO (\delta\rho) \varphi dx + \intO\big(\nu\nabla(\delta\rho)+(\delta\rho)B[u_2]\big)\cdot \nabla\varphi dx ={-\intO \rho_1  b\otimes (u_2-u_1)}\cdot \nabla\varphi dx.
\]
The latter is equivalent to the linearized equation \eqref{lineareq2} for $z = \delta \rho $ and $v = u_2-u_1.$ Applying the estimate \eqref{estimatez} yields the desired result.
\end{proof}

Notice that \eqref{estdeltarho} implies an analogous inequality for $\|\rho_2-\rho_1\|_{C([0,T];\spaceH)}$
since $W(0,T)$ is continuously embedded in $C([0,T];\spaceH)$. The constant  neither depends on
$\|u_2\|_\infty$ as well, since the embedding constant does not depend on the control.

%%%%%%%%%%%%%%%%%%%%%%%%%%%%%%%%%%%%%%%%%%%%%%%%%%%%%%%
\section{The adjoint equation} \label{S5}
\setcounter{equation}{0}
%%%%%%%%%%%%%%%%%%%%%%%%%%%%%%%%%%%%%%%%%%%%%%%%%%%%%%%

%%%%%%%%%%%%%%%%%%%%%%%%%%%%%%%%%%%%%%%%%%%%%%%%%%%%%%%
\subsection{Definition of the adjoint equation}
%%%%%%%%%%%%%%%%%%%%%%%%%%%%%%%%%%%%%%%%%%%%%%%%%%%%%%%

By an adjoint state, the variational inequality \eqref{FirstAbstract} can be transformed to a more convenient form. To this aim, we introduce the following adjoint equation 
for the {\em adjoint state} $p$ associated with $(\rho,u)$:
\begin{align}
\label{adjoint}
- \partial_t p - \nu \Delta p + B[u]\cdot \nabla p &= {\alpha_Q}(\rho-\rho_Q)\quad \text{in } Q,\\
 \label{adjointOmega} p(T) &={\alpha_\Omega}( \rho(T)-\rho_\Omega)\quad \text{ in } \Omega,\\
\label{adjointSigma} \partial_np &= 0\quad \text{on } \Sigma.
\end{align}
The form of the adjoint equation can be found {\em e.g.} by application of a formal Lagrangian technique,
{\em cf.} \cite{Troltzsch2010}[chpt. 2.6]. 
 In weak formulation, the adjoint equation at $(\rho,u)$ is defined by 
\be 
\label{adjointeq}
\begin{split}
 \intO  \Big(- \varphi \, \partial_tp + \nu \nabla p \cdot \nabla \varphi  + {\varphi }\, B[u]\cdot \nabla p \Big)dx &= {\alpha_Q} \intO (\rho-\rho_Q)\varphi dx\\
 &\qquad\qquad \text{on } \mathcal{D}'(0,T)\,\,\text{for all } \varphi\in H^1(\Omega),\\
  p(T) &= {\alpha_\Omega} (\rho(T)-\rho_\Omega).
  \end{split}
 \ee
 The unique (weak) solution 
$p$ of \eqref{adjoint}-\eqref{adjointSigma} is called the {\em adjoint state} associated with $(\rho,u)$.
Note that, with  $f:= {\alpha_Q}(\rho-\rho_Q),$ \eqref{adjointeq} can be rewritten as
 \be
 \label{adjointeqgen}
\begin{split}
  -\partial_t p + a[u(\cdot)](\cdot,p) &= f\quad \text{in } L^2(0,T;\spaceVp),\\
  p(T) &= {\alpha_\Omega}(\rho(T)-\rho_\Omega),
  \end{split}
 \ee
 provided that $p$ enjoys the higher regularity $p \in W(0,T)$.

For the weak formulation we will also use the definition of Ladyzhenskaya {\em et al.} \cite{Ladyzhenskaya_etal1968}, because for given $u \in \Utwo$ we were only able to show the regularity $p \in W^{1,0}_2(Q)$, namely
\be 
\label{adjointeqL}
\begin{split}
 {\iint_Q} & \Big(p\, \partial_t \varphi  + \nu \nabla p \cdot \nabla \varphi  + {\varphi }\, B[u]\cdot \nabla p \Big)dxdt = \\
 & - \intO  {\alpha_\Omega} (\rho(T)-\rho_\Omega) \,\varphi(T) dx + {\alpha_Q} \iint_Q (\rho-\rho_Q)\varphi dxdt\quad \\
&\qquad\qquad \qquad \text{for all } \varphi\in {W^{1,1}_2(Q) \text{ with } \varphi(\cdot,0)}=0.
  \end{split}
 \ee
 
 \begin{proposition} \label{P5.1}
 Given $u \in \Uinfty$ and $f\in L^2(0,T;\spaceVp)$,  equation \eqref{adjoint}-\eqref{adjointSigma}, {with right-hand side $\alpha_Q(\rho-\rho_Q)$ of \eqref{adjoint} replaced by the general function $f$,}  has a unique weak solution $p \in W(0,T)$ and the following estimate holds, 
 \be
\|p\|_{W(0,T)} \leq C \big( \| \alpha_\Omega(\rho(T)-\rho_\Omega ) \|_{L^2(\Omega)} + \|f\|_{L^2(0,T;\spaceVp)}  \big),
 \ee
 where {$\rho := G(u)$ and} $C$ depends continuously on $\|u\|_2$.
 In particular,{when $f= {\alpha_Q}(\rho-\rho_Q)$}, one has
  \be
  \label{pest}
\|p\|_{W(0,T)} 
 \leq C \big(\|  \alpha_\Omega( \rho(T) - \rho_\Omega )\|_{L^2(\Omega)} + 
\| \alpha_Q(\rho-\rho_Q) \|_{L^2(Q)}  \big).
 \ee
\end{proposition}
 \begin{proof}
We first apply the transformation of time $\tau = T - t$ and $\tilde p(\tau) = p(T-\tau)$. Then the equation is transformed to a forward one. In particular, the terminal condition for $p$ becomes an initial condition for $\tilde p$. 
Then we proceed as in the proof of   Theorem \ref{ThmExistSol}. It is easy to confirm that the estimates $\eqref{estarhophi}$ and $\eqref{acoercive}$ 
remain true for the choice $\rho := \varphi$ and $\varphi:= p$, where $\varphi \in L^2(0,T,H^1(\Omega))$ is 
the test function and $p \in L^2(0,T,H^1(\Omega))$ is the desired solution. Notice that $p$ appears through a gradient. Now the existence of $p \in W(0,T)$ follows again from the result by Lions and Magenes \cite{LioMag68a}.
\end{proof}

\begin{remark}
For $u \in L^2(0,T;\mathbb{R}^n)$, we can show the existence of an adjoint state $p \in W^{1,0}_2(Q)$
that satisfies the weak formulation \eqref{adjointeqL}. To do this, we proceed as in the proof of 
Theorem \ref{u2-case} and approximate $u$ by a sequence of controls   $\{u_k\}\subset \Uinfty$. An
associated subsequence of adjoint states $p_k$ converges weakly in $ W^{1,0}_2(Q)\cong L^2(0,T;H^1(\Omega))$ to some $p \in W^{1,0}_2(Q)$, hence  $\nabla p_k$ converges weakly in $L^2(0,T;L^2(\Omega)^n)$ to $\nabla p$. Again, the bilinear term is the delicate point: as in \eqref{adjointeqL}, we are allowed to use more regular test functions $\varphi \in W^{1,1}_2(Q) \subset C([0,T];L^2(\Omega))$. We deduce 
\begin{equation} \label{auxp}
\iint_Q \varphi \, u_k\otimes b \cdot \nabla p_k\, dxdt \to \iint_Q \varphi \, u \otimes b \cdot \nabla p \, dxdt, \quad k \to \infty.
\end{equation}
Here,  we take advantage of the weak convergence $u_k\otimes b \cdot \nabla p_k \rightharpoonup  u \otimes b \cdot \nabla p$ in $L^1(0,T;L^2(\Omega))$ that fits to the regularity $\varphi \in C([0,T];L^2(\Omega))$. \\
We were not able to prove that $p$ belongs to $W(0,T)$. The reason is that, in contrast to what occurs for the state equation, here the mapping 
\begin{equation} \label{bilinear2}
\varphi \mapsto  \iint_Q \varphi \, u \otimes b \cdot \nabla p \, dxdt
\end{equation}
is not continuous on $L^2(0,T;H^1(\Omega))$ for  $u \in \Utwo$; notice that we only know $\nabla p \in L^2(Q;\cR^n)$. Therefore,  this linear functional does not belong to $L^2(0,T;H^1(\Omega)^*)$.

We should mention that, even for $p \in W(0,T)$, the integral \eqref{bilinear2} is not well defined with test functions $\varphi \in L^2(0,T;H^1(\Omega))$,
since we only know that $\nabla p \in L^2(0,T;L^2(\Omega)^n)$. We would need the additional regularity 
$\nabla p \in L^\infty(0,T;L^2(\Omega)^n)$, if $u \in \Utwo$. We did not try to prove this, since our controls
are essentially bounded. Therefore, in what follows we concentrate on the case of bounded controls. The main result of our paper, the second-order sufficient optimality conditions, is not influenced by this restriction.
For a similar setting but for infinite horizon control, in  \cite[Prop. 4.8]{Breiten_Kunisch_Pfeiffer2018_sicon}
the authors were able to show that $p$ is in $W(0,T)$.
\end{remark}

%%%%%%%%%%%%%%%%%%%%%%%%%%%%%%%%%%%%%%%%%%%%%%%%%%%%%%%%%%%%%%%%%%%%%%%%%%%%%%%%%%%%%%%%%%%%%%%%%%%%%%%%%%%%%%
\subsection{First order necessary optimality conditions in terms of the adjoint state}
%%%%%%%%%%%%%%%%%%%%%%%%%%%%%%%%%%%%%%%%%%%%%%%%%%%%%%%%%%%%%%%%%%%%%%%%%%%%%%%%%%%%%%%%%%%%%%%%%%%%%%%%%%%%%%
In this subsection we rewrite the first order condition of Proposition \ref{PropFirst} in terms of the adjoint state. For this, we show the following technical result. 
\begin{lemma}
Assume that  $u \in L^\infty(0,T;\cR^n)$ is given and let $\rho=G(u)$ be its associated state. 
Let $z$ be the weak solution of the linearized state equation \eqref{lineq} corresponding to $v\in \Utwo,$ and let $p \in W(0,T)$ be the weak solution of the adjoint equation \eqref{adjoint}-\eqref{adjointSigma}. Then
\be
\label{zp}
{\alpha_Q}\intT\intO(\rho-\rho_Q)z dx dt + {\alpha_\Omega}\intO(\rho(T)-\rho_\Omega)z(T)dx = -\intT \intO\rho( b\otimes v)\cdot \nabla p dx dt.
\ee
\end{lemma}

\if{
or, equivalently,
\begin{multline}
\intT\intO(\rho-\rho_Q)z dx dt + \intO(\rho(T)-\rho_\Omega)z(T)dx \\= \intT \intO p\, {\rm div}(\rhob( b\otimes v)) dt - \intT\int_{\Gamma} p\rhob( b\otimes v) \cdot n ds dt\\
=\intT \intO p\, {\rm div}(\rhob( b\otimes v)) dt + \intT\int_{\Gamma} p(\nu\nabla z+z B[\ub]) \cdot n ds dt,
\end{multline}
where the second equality follows from the boundary condition of $z.$
}\fi

\begin{proof}
It follows by testing the linearized equation \eqref{lineareq2} with $\varphi:=p$ and the adjoint equation \eqref{adjointeq} with $\varphi:=z,$ and subsequent integration by parts (a detailed proof can be found in \cite{AronnaTroeltzsch2020v1}).
%We write down the weak formulation of the linearized equation \eqref{lineareq2} with test function $\varphi := p$, where $p$ is the adjoint state associated with $u$, i.e. the solution of the adjoint equation \eqref{adjointeq}. Moreover, we insert in the adjoint equation \eqref{adjointeq} the test function $\varphi := z$. This yields
% \begin{gather}
% \label{ztp}
%  \int_0^T \left\langle \partial_t z , p \right\rangle\, dt + \int_0^T \intO \Big(\nu\nabla z\cdot \nabla p +zB[u]\cdot \nabla p\Big) {dxdt}= -  \int_0^T \intO \rho ( b\otimes v)\cdot \nabla p dxdt\\
%\label{ptz} 
%\int_0^T \left\langle - \partial_t p , z \right\rangle\, dt + \int_0^T \intO \Big(\nu \nabla p \cdot \nabla z  + zB[u]\cdot \nabla p \Big)dxdt = {\alpha_Q}\int_0^T \intO (\rho-\rho_Q)\,zdxdt.
%\end{gather}
%Here and below $u \in L^\infty(0,T;\cR^n)$ yields
%$p \in W(0,T)$ so that all terms are well defined and integration by parts is possible.
%Performing this in \eqref{ptz},  invoking the conditions $z(0) = 0$, and
%$p(T) = {\alpha_\Omega}(\rho(T)-\rho_\Omega)$, equation \eqref{ptz} is transformed to 
%\[
%\begin{aligned}
% \int_0^T\left\langle p,\partial_t z \right\rangle dt+ \int_0^T \intO \Big(\nu \nabla p \cdot \nabla z  + zB[u]\cdot \nabla p\Big)dxdt  =  {\alpha_\Omega}  \int_\Omega (\rho(T)-\rho_\Omega) z(T)\,dx + {\alpha_Q} \int_0^T \intO (\rho-\rho_Q)\,zdxdt.
%\end{aligned}
%\]
%The left-hand side of latter equation is equal to the left-hand side of \eqref{ztp}. Therefore,
%also the associated right-hand sides must be equal. This confirms the claim of the lemma.
\end{proof}

Let us introduce the notation
\be
\label{Phi}
\Phi_i(t):=-\intO \rho(t)  b_i{\frac{\partial {p}(t)}{\partial x_i}}\,dx + \gamma_i u_i(t)+ \beta_i \quad \text{for } i=1,\dots,n.
\ee
\begin{theorem}
For any $u \in \Uinfty$ and $v \in \Utwo$, one has
\be
\label{F'exp}
F'(u)v=\sum_{i=1}^n \intT \Phi_i(t) v_i(t)dt,
%F'(u)v=\sum_{i=1}^n \intT \Big[-\intO \rho  b_i\frac{\partial p}{\partial x_i} dx + \gamma_i u_i+ \beta_i  \Big]v_i\,dt,
\ee
where $\Phi$ is defined in \eqref{Phi}, with $\rho = G(u)$ and $p\in W(0,T)$ being the associated adjoint state.
\end{theorem}

\begin{proof}
Let $z$ be the solution of the state equation linearized at $(\rho,u).$
Then one has
\be\label{F'}
\begin{split}
F'(u)v 
& = {\alpha_Q}  \intT\intO(\rho-\rho_Q)z dx dt + {\alpha_\Omega}  \intO(\rho(T)-\rho_\Omega)z(T)dx \\
&\qquad+\sum_{i=1}^n \intT \left(\gamma_iu_iv_i+ \beta_i v_i\right)dt\\
&= -\intT \intO\rho( b\otimes v)\cdot \nabla p dx dt+\sum_{i=1}^n \intT \left(\gamma_i u_iv_i+ \beta_i v_i\right)dt\\
&=\sum_{i=1}^n \intT \Big[-\intO \rho  b_i\frac{\partial p}{\partial x_i} dx + \gamma_i u_i+ \beta_i \Big]v_i\,dt,
\end{split}
\ee
where we used \eqref{zp} in the second equality. This proves the result.
\end{proof}

By means of the expression \eqref{F'exp}, we can reformulate the first order optimality condition of Proposition \ref{PropFirst} as follows:
\begin{corollary}
If $\ub$ is an $L^\infty$-local minimum for \eqref{P}, then
\be
\label{First}
\sum_{i=1}^n \intT \bar\Phi_i(t)\big(u_i(t)-\ub_i(t)\big)\,dt \geq 0\quad \text{for all } u\in \mathcal{U}_{\rm ad},
%\sum_{i=1}^n \intT \Big[-\intO \rhob  b_i\frac{\partial \bar{p}}{\partial x_i} dx + \gamma_i \ub_i+ \beta_i  \Big](u_i-\ub_i)\,dt \geq 0,\quad \text{for all } u\in \mathcal{U}_{\rm ad},
\ee
where {$\bar\Phi$ is the function given in \eqref{Phi} associated to $\ub$.
Consequently, we have
\be
\label{signPhi}
\left\{
\begin{split}
& \ds \bar{\Phi}_i(t) > 0   \Longrightarrow  \ub_i(t) = u_i^{\min}(t),\\
& \ds \bar{\Phi}_i(t) < 0   \Longrightarrow  \ub_i(t) = u_i^{\max}(t),\\
& \ds u_i^{\min}(t) < \ub_i (t) < u_i^{\max}(t)  \Longrightarrow \bar{\Phi}_i(t) =0,
\end{split}
\right.
\ee
a.e. on $[0,T]$ and for all $i=1,\dots,n.$}
\end{corollary}

%%%%%%%%%%%%%%%%%%%%%%%%%%%%%%%%%%%%%%%%%%%%%%%%%%%%%%%
\section{Second order analysis}\label{S6}
%%%%%%%%%%%%%%%%%%%%%%%%%%%%%%%%%%%%%%%%%%%%%%%%%%%%%%%
\setcounter{equation}{0}

The optimal control problem is a non-convex one, hence first order necessary optimality conditions should be complemented by a second order analysis. Second order sufficient optimality conditions serve as important assumption for the numerical analysis. For instance, the stability of locally optimal  solutions under a numerical approximation of the problem or the convergence analysis of numerical methods such as SQP or semismooth Newton techniques need second order sufficient optimality conditions as hypothesis. Though it is hardly possible to confirm them numerically, they are used as assumption for the analysis. This is similar to constraint qualifications in nonlinear optimization that can be verified only in exceptional cases but are indispensable for the analysis.

To establish second order optimality conditions, we will apply general results by Casas and Tr\"oltzsch,   \cite[Theorems 2.2 and 3.3]{CasasTroeltzsch2012}. 
For this purpose, we have to verify that the conditions {\rm \ref{C1}}-{\rm \ref{C3}} below  are satisfied for problem \eqref{P}. More precisely, we will need {\rm \ref{C1}} in the second order necessary condition of Theorem \ref{SONC} below, while {\rm \ref{C1}}-{\rm \ref{C3}} are used in the sufficient one of Theorem \ref{SOSC}.

%%%%%%%%%%%%%%%%%%%%%%%%%%%%%%%%%%%%%%%%%%%%%%%%%%%
\subsection{Second order conditions for an optimization problem in Banach spaces}
%%%%%%%%%%%%%%%%%%%%%%%%%%%%%%%%%%%%%%%%%%%%%%%%%%%%%

We consider a Banach space $U_\infty$ and a Hilbert space $U_2,$ endowed with the norms $|\cdot|_\infty$ and $|\cdot|_2$, respectively, and such that $U_\infty$ is continuously embedded in $U_2.$ Let us introduce the abstract optimization problem
\be
\label{calP}\tag{$\mathcal{P}$}
\min_{u\in \mathcal{K}} \mathcal{J}(u),
\ee
where $\mathcal{K}\subseteq U_\infty$ is a given nonempty convex set and $\mathcal{J} \colon \mathcal{A}\to \cR$ is the objective function, defined  and twice continuously differentiable in an open subset $\mathcal{A} \subset U_\infty$ that covers $\mathcal{K}.$
We say that $\ub$ is a {\em $U_\infty$-local solution} of \eqref{calP} if there exists $\eps>0$ such that $\mathcal{J}(\ub) \leq \mathcal{J}(u)$ holds for all $u\in \mathcal{K} \cap \{u\in U_\infty:  |u-\ub|_\infty<\eps\}.$
\\
\indent If $\ub$ is a $U_\infty$-local solution of \eqref{calP}, then the following first order necessary condition is satisfied:
\be
\label{FONCJ}
\mathcal{J}'(\ub)(u-\ub)\geq 0 \quad \text{for all } u\in \mathcal{K}.
\ee

 Let us fix $\ub$ in $\mathcal{K}$. We consider the following conditions for problem \eqref{calP}. All the notions of differentiability of $\mathcal{J}$ are to be understood in the sense of $U_\infty.$
\begin{itemize}
\item[\namedlabel{C1}{(C1)}] The functional $\mathcal{J}$ is of class $C^2$ in 
$\mathcal{A}$. For every $u \in \mathcal{K},$ there exist continuous extensions
\[
\mathcal{J}'(u) \in \mathcal{L}(U_2;\cR),\quad \mathcal{J}''(u)  \in \mathcal{B}(U_2;\cR)
\]
of $\mathcal{J}'(u)$ and $\mathcal{J}''(u)$, where $\mathcal{B}(U_2;\cR)$ denotes the Banach space of continuous bilinear real functionals on  $U_2\times U_2$.
\item[\namedlabel{C2}{(C2)}] For any sequence $\{ (u_k,v_k)\} \subset \mathcal{K} \times U_2 $ with $u_k\to \ub$ in $U_2$ and $v_k \rightharpoonup v$ weakly in $U_2$ as $k \to \infty$, there holds
\begin{gather}
\label{limF'}
\mathcal{J}'(\ub) v = \lim_{k\to\infty} \mathcal{J}'(u_k)v_k.
\end{gather}
\item[\namedlabel{C3}{(C3)}] For any sequence defined as in {\rm (C2),} the following two properties are satisfied for some $\Lambda>0:$ it holds
\begin{gather}
\label{liminfD2J} \mathcal{J}''(\ub) v^2 \leq \liminf_{k\to \infty} \mathcal{J}''(u_k) v_k^2,\\
\label{liminfD2J0} \text{and, if } v=0, \text{ then } \,\,\, \Lambda \liminf_{k\to\infty} |v_k|_2^2 \leq \liminf_{k\to \infty} \mathcal{J}''(u_k)v_k^2.
\end{gather}
\end{itemize}

For a fixed control $\ub \in \mathcal{K}$, let us define the following sets
\begin{equation}
\label{cones}
 \begin{split}
S(\ub)&:= \big\{v\in U_\infty: v=\lambda(u-\ub)\, \text{for some } \lambda>0 \text{ and } u\in \mathcal{K} \big\},\\
C(\ub)&:= {\rm cl}_{U_2}(S(\ub)) \cap\{ v\in U_2: \mathcal{J}'(\ub)v=0\},\\
D(\ub) &:= \big\{v\in S(\ub) : \mathcal{J}'(\ub)v=0 \big\}.
\end{split}
\end{equation}
The set $S(\ub)$ is called {\em cone of feasible directions}, while $C(\ub)$ is the {\em critical cone.}
\\
\vspace{1.5ex}
\\
\indent 
We first state the general Theorems 2.2 and 2.3 from \cite{CasasTroeltzsch2012}, that we will apply to obtain the second order necessary and sufficient optimality conditions for the control of our Fokker-Planck equation in Theorems \ref{SONC} and \ref{SOSC} below.

\begin{theorem}[Casas-Tr\"oltzsch \cite{CasasTroeltzsch2012}]
\label{SONC2012}
Let $\ub$ be a $U_\infty$-local solution for \eqref{calP}. Assume that {\rm \ref{C1}} and the {\em regularity condition} $C(\ub) = {\rm cl}_{U_2}D(\ub)$ hold. Then
\[
\mathcal{J}''(\ub)v^2 \geq 0\quad \text{for all } v\in C(\ub).
\]
\end{theorem}

\begin{theorem}[Casas-Tr\"oltzsch \cite{CasasTroeltzsch2012}]
\label{SOSCCT}
Suppose that {\rm \ref{C1}-\ref{C3}} are fulfilled for problem \eqref{calP}. Let $\ub \in \mathcal{K}$ satisfy  the 
first order necessary condition \eqref{FONCJ} along with
\be
\mathcal{J}''(\ub)v^2 >0\quad \text{for all } v \in C(\ub) \backslash \{0\}.
\ee
Then, there exist $\eps>0$ and $\delta >0$ such that
\be
\mathcal{J}(\ub) + \frac{\delta}{2} |u-\ub|_2^2 \leq \mathcal{J}(u)\quad  \text{for all } u \in \mathcal{K} \cap B_{2,\eps}(\ub),
\ee
where $B_{2,\eps}(\ub)$ is the open ball in $U_2,$ of radius $\eps$ and centered in the origin.
\end{theorem}

\begin{remark}
The two theorems above were formulated for problems where the so-called {\em two-norm discrepancy} occurs.
This means first that the objective functional $\mathcal{J}$ is not of class $C^2$ in $U_2$, while it is $C^2$ in $U_\infty$. Second, it includes that the coercivity $\mathcal{J}''(\bar u)v^2 \ge \delta \, |v|_\infty^2$ cannot be shown for the $U_\infty$-norm for any $\delta>0$, while it can possibly be fulfilled with the norm $|\cdot|_2$
of $U_2$. We refer e.g. to Ioffe, \cite{Ioffe1979}.
In our optimal control problem, the two-norm discrepancy does not occur, since the reduced objective functional $F$ belongs to the class  $C^2$ in $U_2$. We might work with $U_\infty = U_2 = \Utwo$.
\end{remark}

\vspace{1.5ex}
In the remainder of this paper, we will confirm the three conditions {\rm \ref{C1}-\ref{C3}} for our optimal control problem \eqref{P}.
For this purpose, we consider 
\[
U_\infty=\mathcal{A}:= \Uinfty,\quad U_2:= \Utwo, \quad \mathcal{J}:= F,\quad \mathcal{K}:=\Uad.
\]
The conditions {\rm \ref{C1}-\ref{C2}} are obviously satisfied if one had $U_\infty=U_2 = 
\Utwo$. However, for the confirmation of \ref{C3}, we need higher regularity of the state and the adjoint state. To this end, we have to work with bounded controls so that it appeared to be more natural for us to consider a two-norm setting and to perform our analysis in $U_\infty = \Uinfty$ and $U_2 = \Utwo$. Though we might work with  $U_\infty = U_2 = \Utwo$ and simplify the presentation, we decided to keep this two-norm setting to show those readers, who really have to deal with the two-norm discrepancy, the technique of proof.

%%%%%%%%%%%%%%%%%%%%%%%%%%%%%%%%%%%%%%%%%%%%%%%
\subsection{Second derivative of the reduced cost functional}
%%%%%%%%%%%%%%%%%%%%%%%%%%%%%%%%%%%%%%%%%%%%%%%

Let us compute the second derivative $G''(u)[v_1,v_2]$ of the control-to-state operator, for $u\in \Uinfty$ and $v_1,v_2 \in \Utwo$. Notice that the differentiability of $G$ has already been proven in Corollary \ref{CorGCinfty}. We have
\[
\partial_t G(u) - \nu \Delta G(u) - {\rm div} (B[u] G(u)) =0
\]
subject to associated initial and boundary conditions.
Then, differentiating with respect to $v_1 \in \Utwo$, we obtain
\[
\partial_t (G'(u) v_1) -\nu \Delta (G'(u)v_1)-{\rm div}( b\otimes v_1\, G(u)) - {\rm div} (B[u] G'(u)v_1)=0.
\]
Another differentiation with respect to $v_2 \in  \Utwo$ yields
\begin{multline*}
\partial_t (G''(u)[v_1,v_2])- \nu\Delta(G''(u)[v_1,v_2])-{\rm div}( b \otimes v_1 G'(u)v_2)
\\-{\rm div}( b\otimes v_2 G'(u)v_1) - {\rm div}(B[u] G''(u)[v_1,v_2])=0.
\end{multline*}
Setting $z_i := G'(u)v_i,$ for $i=1,2,$ and $w:=G''(u)[v_1,v_2]$ for the unknown second-order derivative, we get the following equation for $w$, written in strong form:
\begin{equation*}
\begin{split}
w_t- \nu\Delta w  - {\rm div}(B[u] w) &= {\rm div} \, (z_1\, b \otimes v_2 + z_2\, b \otimes v_1),\\
(\nu\nabla w + wB[u])\cdot n &= -(z_1\, b \otimes v_2 + z_2\, b \otimes v_1) \cdot n ,\\
w(0)&=0.
\end{split}
\end{equation*}
The boundary condition is also obtained by implicit differentiation.
%{\color{blue} In weak formulation, the equation for $w = G''(u)[v_1,v_2]$ reads
%\begin{equation*}
%\begin{split}
%\int_0^T \Big\langle \partial_t w , \varphi \Big\rangle\, dt &+ \iint_Q ( \nu\nabla w + w B[u]) \cdot \nabla \varphi \, dxdt\\
%&= 
%- \iint_Q (z_1 \, b \otimes v_2 + z_2\, b \otimes v_1) \cdot \nabla \varphi \, dxdt, \quad \forall \varphi \in L^2(0,T;H^1(\Omega)),\\
%%(\nu\nabla w + wB[u])\cdot n &= \tcr{-(z_1\, b \otimes v_2 + z_2\, b \otimes v_1) \cdot n },\\
%w(0)&=0.
%\end{split}
%\end{equation*}}
Notice that $v_1$ and $v_2$ depend only on $t$. Moreover,  we have $b \in L^{\infty}(\Omega;\cR^n)$ and $z_i \in W(0,T)$, for 
$i = 1,2$. In particular, $z_i$ belongs to $C([0,T];L^2(\Omega))$, for 
$i = 1,2$.
Therefore, $z_1\, b \otimes v_2 + z_2\, b \otimes v_1$ are in $L^2(Q;\cR^n)$. The homogeneous initial condition follows obviously by differentiating the equation
$G(u)(\cdot,0) = \rho_0(\cdot)$ twice with respect to $u$.
The associated weak formulation is
\begin{align}
\label{varforw} 
\frac{d}{dt} ( w, \varphi) + a[u](w,\varphi) &=\big(d[z_1 b\otimes v_2] + d[z_2 b\otimes v_1]\big)(\varphi)\quad \text{on } \mathcal{D}'(0,T)\,\, \text{for all } \varphi \in \spaceV,\\
\label{initcondw} w(0) &= 0\quad \text{in } \spaceH.
\end{align}

\begin{proposition}
Let $u \in \Uinfty$ be given and $v_1,v_2 \in \Utwo$ be control increments with associated linearized states $z_1,z_2$, respectively.
Then the second order derivative of the reduced cost functional $F$ at $u$ in the direction pair $(v_1,v_2)$  is given by
\begin{multline}
\label{F''}
F''(u)[v_1,v_2] = \intT\intO \Big({\alpha_Q} z_1z_2   - \nabla p \cdot ( z_2 b \otimes v_1 +z_1 b\otimes v_2 )\Big)dxdt 
\\+ \intT\sum_{i=1}^n \gamma_i v_{1,i} v_{2,i} dt + {\alpha_{\Omega}} \intO z_1(T)z_2(T) dx,
\end{multline}
where $p \in W(0,T)$ is the adjoint state associated with $u$.
\end{proposition}

\begin{proof}
The second order derivative of $F$ is computed by the chain rule. One has
\begin{multline*}
F'(u)v_1={\alpha_{Q}}\big(G(u)-\rho_Q,G'(u)v_1\big)_{L^2(Q)} \\ +{\alpha_{\Omega}}\big((G(u))(T)-\rho_\Omega,(G'(u)v_1)(T)\big)_{L^2(\Omega)}
+\sum_{i=1}^n \Big( \gamma_i(u_i,v_{1,i})_{L^2(0,T)} + \beta_i\intT v_{1,i} dt \Big).
\end{multline*}
Then, differentiating this expression w.r.t. $v_2$ yields
\begin{equation}
\label{F''exp}
\begin{split}
F''(u)[v_1,v_2] = & \,{\alpha_Q} \big( G'(u)v_2,G'(u)v_1\big)_{L^2(Q)} + {\alpha_Q} \big( G(u)-\rho_Q,G''(u)[v_1,v_2]\big)_{L^2(Q)} \\
& + {\alpha_\Omega} \big((G'(u)v_2)(T),(G'(u)v_1)(T)\big)_{L^2(\Omega)} \\
&+ {\alpha_\Omega} \big((G(u))(T)-\rho_\Omega,(G''(u)[v_1,v_2])(T)\big)_{L^2(\Omega)}+\sum_{i=1}^n \gamma_i(v_{2,i},v_{1,i})_{L^2(0,T)}\\
=&\, {\alpha_Q}(z_1,z_2)_{L^2(Q)} + {\alpha_Q}(G(u)-\rho_Q,w)_{L^2(Q)} + {\alpha_\Omega}(z_1(T),z_2(T))_{L^2(\Omega)}\\
&+{\alpha_\Omega}(p(T),w(T))_{L^2(\Omega)}+\sum_{i=1}^n \gamma_i(v_{1,i},v_{2,i})_{L^2(0,T)},
\end{split}
\end{equation}
where we used that $z_i = G'(u)v_i,$ for $i=1,2$, and $w=G''(u)[v_1,v_2]$. Moreover, we invoked the terminal condition for $p$ from the adjoint equation \eqref{adjoint}.
This form of $F''$ includes the solution $w$ that implicitly depends on the increments $v_1,v_2$ via
the partial differential equation \eqref{varforw}. Now we proceed in the same way that we used to show
equation \eqref{zp}. We insert $p$ as test function in equation \eqref{varforw} and  $w$ as test function in the adjoint equation \eqref{adjointeq}. After some integration by parts, we arrive at the relation
\begin{equation*}
{\alpha_Q} \big(G(u)-\rho_Q,w\big)_{L^2(Q)} + {\alpha_\Omega}\big(p(T),w(T)\big)_{L^2(\Omega)}=- \Big(\nabla p \, , \, b \otimes v_1 z_2
+ b\otimes v_2 z_1\Big)_{L^2(Q)}
\end{equation*}
that yields \eqref{zp}. Inserting this result in \eqref{F''exp}, we verify the claim.
\end{proof}

We will state our second order optimality conditions in terms of the {\em quadratic form}
\be
\label{D2F}
F''(u)v^2 = \iint_Q \Big[{\alpha_Q}z^2-2\nabla p\cdot (z b \otimes v)\Big]dxdt + \sum_{i=1}^n \gamma_i\intT v_i^2 dt + {\alpha_{\Omega}} \intO z(T)^2 dx,
\ee
that is obtained from the general form \eqref{F''} of $F''$ by the choice $v_1=v_2=v \in \Utwo$. For all $u \in U_\infty=\Uinfty$, the quadratic form is defined and continuous in $U_2=\Utwo$.

%%%%%%%%%%%%%%%%%%%%%%%%%%%%%%%%%%%%%%%%%%
\subsection{Critical directions}
%%%%%%%%%%%%%%%%%%%%%%%%%%%%%%%%%%%%%%%%%%

For the sets defined in \eqref{cones}, the following characterization is easily obtained.
\begin{proposition}[Characterization of the critical cone]
\be
\label{CharaCu}
C(\ub)= \left\{
\begin{split}
&v\in \Utwo: \,\text{a.e. on } [0,T]  \text{ and for } i=1,\dots,n,\\
&v_i
\left\{
\begin{split} 
\geq 0\quad &\text{if } \ub_i=u^{\min}_i\\
  \leq 0\quad &\text{if } \ub_i=u^{\max}_i\\
 =0\quad &\text{if } \bar\Phi_i\neq 0
\end{split}
\right.
%\quad\text{a.e. on } [0,T],  \text{ for } i=1,\dots,n,\\
\end{split}
\right\},
\ee
where $\bar\Phi$ was defined in \eqref{Phi}.
\end{proposition}

\begin{proof}
We follow essentially the lines of the proof in \cite[p. 273]{CasasTroeltzsch2012}.
Let us use $K(\ub)$ to denote the set on the r.h.s. of \eqref{CharaCu}.

 First take $v\in C(\ub)$. Then $F'(\ub)v=0$ and $v$ is in $C(\ub)$ so by definition $v\in C(\ub)$, we used this below, we do not deduce it there exists $\{v_k\} \subseteq S(\ub)$ 
such that $v_k \to v$ in $\Utwo.$ By definition of $S(\ub)$ one has, necessarily, that  $v_{k,i}\geq0$ if $\ub_i = u^{\min}_i$ and $v_{k,i} \leq0$ if $\ub_i=u^{\max}_i,$ a.e. on $[0,T]$ and for every $i=1,\dots,n.$
Clearly, this property is preserved for the limit in $\Utwo,$ so that it also holds for $v.$
%$v_{i}\geq0$ if $\ub_i = u^{\min}_i$ and $v_{i} \leq0$ if $\ub_i=u^{\max}_i.$ 
From this fact,  due to the expression of $F'$ given in \eqref{F'exp} and the first order necessary condition of \eqref{signPhi}, we deduce that
$$
0=F'(\ub)v = \sum_{i=1}^n \intT \bar\Phi_i(t)v_i(t) dt = \sum_{i=1}^n \intT \left| \bar\Phi_i(t)v_i(t) \right|dt,
$$ 
which implies that $v_i=0$ if $\bar\Phi_i \neq 0.$ This proves that $C(\ub) \subseteq K(\ub)$.

In order to prove the converse inclusion, take $v\in K(\ub)$, and define, for each positive integer $k$ and for each $i=1,\dots,n,$
\be
\label{vki}
v_{k,i}(t):=
\left\{
\begin{split}
&0 \qquad \text{if } u^{\min}_i(t)<\ub_i(t)<u^{\min}_i(t)+\frac1k \, \text{ or } \, u^{\max}_i(t)-\frac1k < \ub_i(t) < u^{\max}_i(t),\\
&\mathbb{P}_{[-k,k]}(v_i(t))\qquad \text{otherwise},
\end{split}
\right.
\ee
where $\mathbb{P}_{[-k,k]}$ denotes the pointwise projection onto the interval $[-k,k]$.
Thus $v_k\in \Uinfty,$ $v_k \to v$ in $\Utwo$ and it easily follows that $u_k:=\ub+\lambda_k v_k$ belongs to $\Uad$ for 
$\displaystyle\lambda_k := \frac{1}{k^2};$
% $$ \displaystyle\lambda_k:= \min\left\{ \frac{1}{k^2},\min_{i=1,\dots,n}\, \frac{u^{\max}_i(t)-u^{\min}_i(t)}{k} \right\} . $$ 
% If I'm not wrong, there is no need to put the second term in the minimum. We only have to worry when $u^{\min}_i(t) + \frac1k < \ub_i(t) < u^{\max}_i-\frac1k,$ and in this case, $u^{\min}_i(t) \leq  \ub_i \pm k \frac{1}{k^2} \leq u^{\max}_i$, then  $u^{\min}_i(t) \leq  \ub_i \pm v_{k,i} \lambda_k \leq u^{\max}_i.$
hence $v_k \in S(\ub).$ Thus $K(\ub) \subseteq  {\rm cl}_{L^2}(S(\ub)) $. Finally, we observe that by definition of $C(\ub)$ and the expression \eqref{F'exp} of $F'$, one has that $F'(\ub)v=0.$ This shows that $v\in C(\ub)$ and concludes the proof.
\end{proof}

%%%%%%%%%%%%%%%%%%%%%%%%%%%
\subsection{Second order necessary optimality conditions}
%%%%%%%%%%%%%%%%%%%%%%%%%%%

In view of the second order continuous differentiability of $F$ in $\Utwo$, the conditions 
 {\rm \ref{C1}} and {\rm \ref{C2}} are trivially satisfied for our control problem. Therefore, we can directly
apply the general Theorem \ref{SONC2012} and set up second order necessary conditions for the control of our Fokker-Planck equation.

\begin{theorem}[Second order necessary optimality condition]
\label{SONC}
Let $\ub$ be an $L^\infty$-local solution for \eqref{P}. Then
\be
\label{inequalityNC}
F''(\ub)v^2 \geq 0\quad \text{for all } v\in C(\ub).
\ee
\end{theorem}

\begin{proof}
In view of Theorem \ref{SONC2012}, it only remains to prove that $C(\ub) = {\rm cl}_{L^2}D(\ub).$ The inclusion $C(\ub) \supseteq {\rm cl}_{L^2}D(\ub)$ follows easily from the expression of $F'$ given in \eqref{F'exp}, since $\bar\Phi$ belongs to
$\Utwo$.
Let us show the converse inclusion, $C(\ub) \subseteq {\rm cl}_{L^2}D(\ub).$ Take $v\in C(\ub)$, and define $v_k,$  for each positive integer $k$, as done in \eqref{vki}.
Hence, we have $v_k \in S(\ub).$ By definition, we know that  $v_k \to v$ strongly in $\Utwo$.
Moreover, since $|v_{k,i}(t)|\leq |v_i(t)|$ a.e. on $[0,T]$ for every $i=1,\dots,n,$ we have in addition that 
$|\bar\Phi_i(t)v_{k,i}(t)|\leq |\bar\Phi_i(t)v_{i}(t)|=0$ a.e. on $[0,T].$
This implies $F'(\ub)v_k=0.$ Consequently, $v_k\in D(\ub)$ holds for every $k$ and, hence, $v\in {\rm cl}_{L^2}D(\ub).$ This yields the desired inclusion.

Finally, the necessary condition \eqref{inequalityNC} follows from Theorem \ref{SONC2012}.
\end{proof}

\if{ %% old result checking conditions (C1) and (C2)

\begin{proposition}
\label{PropA} The conditions {\rm \ref{C1}} and {\rm \ref{C2}} are fulfilled for problem  \eqref{P}.
\end{proposition}

\begin{proof}
 Recall the expressions for $F'$ and $F''$ given in \eqref{F'} and \eqref{D2F}.
 Let us first verify {\rm \ref{C1}}.
In view of the representation \eqref{F'exp},  the continuous extension of $F'$ is possible if  $ \Phi$
in \eqref{Phi} belongs to $\Utwo$. Obviously, this holds true if the functions $ \rho \, b_i \partial_{x_i} p$ belong to $L^2(0,T;L^1(\Omega))$, for each $1,\dots,n.$ This, however follows from $\rho \in C([0,T];L^2(\Omega))$, $\partial_{x_i} p \in L^2(0,T;L^2(\Omega))$, and $b_i \in L^\infty(\Omega)$.

Next, we confirm the associated extension property of $F''(u)$. First of all, we mention that the mapping $v \mapsto z = G'(u)v$ is continuous from $\Utwo$ to $W(0,T)$, hence also to $C([0,T];L^2(\Omega))\cap L^2(0,T;H^1(\Omega))$. Therefore, the only nontrivial term in the expression \eqref{F''} for $F''$ is the integral
\[
\iint_Q\nabla p \cdot (z_2 b \otimes v_1 + z_1 b\otimes v_2 )\, dxdt.
\]
In view of symmetry, it suffices to consider the term
\[
\iint_Q \nabla p \cdot (b \otimes v_1) z_2\, dxdt,
\]
and to show that the mapping
\[
(v_1,v_2) \longmapsto \iint_Q \nabla p \cdot (b \otimes v_1) z_2\, dxdt
\]
can be continuously extended to $\Utwo \times \Utwo$. 
 This is true, since
 \[
 \begin{aligned}
\left|\iint_Q \nabla p \cdot( b \otimes v_1) z_2\, dxdt\right| &\le  \|\nabla p\|_{L^2(0,T;L^2(\Omega)^n)} \|b\|_{L^\infty(\Omega)^n} \|v_1\|_2 \|z_2\|_{{C([0,T];\spaceH)}} \le C  \, \|p\|_{L^2(0,T;H^1(\Omega))}  \|v_1\|_2 \|v_2\|_2,
 \end{aligned}
 \]
 where we used the estimate \eqref{estimatez} for $z_2$. This is the desired extension property. 
\vspace{1.5ex}

  Now we are going to confirm  {\rm \ref{C2}}. Take a sequence $\{ (u_k,v_k)\} \subset \Uad \times \Utwo$ as it is required in {\rm \ref{C2}}, this is, with $u_k\to \ub$ in $\Utwo$ and $v_k \rightharpoonup v$ (weakly) in $\Utwo$. 
Let us prove that \eqref{limF'} holds.

We have
\be
\label{C2der}
F'(u_k) v_k = \sum_{i=1}^n \intT \underbrace{\left[ \intO \rho_k  b_i \frac{\partial p_k}{\partial x_i} dx + \gamma_i u_{k,i} + b_i \right]}_{\displaystyle\Phi_{k,i}} v_{k,i} dt,
\ee
where $\rho_k := G(u_k)$ and $p_k$ is the solution of \eqref{adjointeq} associated to $(\rho_k,u_k).$
Define  $\rhob:=G(\ub)$, let $\bar p$ denote the corresponding adjoint state  and $\bar\Phi$ be the associated function defined in \eqref{Phi}.
In view of the estimate \eqref{deltap} applied to $p_k-\bar p$, we have  $p_k \to \bar p$ in $C([0,T];L^2(\Omega))$.  From \eqref{estdeltarho},  we get $\rho_k \to \rhob$ in $C([0,T];\spaceH)$, then we deduce  that
$$
\left\|\intO \rho_k(\cdot)  b_i \frac{\partial p_k(\cdot)}{\partial x_i} dx-\intO \rhob(\cdot)  b_i \frac{\partial\bar p(\cdot)}{\partial x_i} dx\right\|_{L^2(0,T)} \to 0
$$
 for each $i=1,\dots,n.$ 
 Therefore, $\Phi_{k}$ given in \eqref{C2der} converges strongly in $\Utwo$ to $\bar\Phi$.
  Consequently, thanks to the convergence of a pairing of weakly-strongly convergent sequences, we can pass to the limit in \eqref{C2der}, and thus \eqref{limF'} follows.
\end{proof}

\deleted{By application of Theorem \ref{SONC2012}, we get the following result.}

}\fi

%%%%%%%%%%%%%%%%%%%%%%%%%%%%%%%%%%%%%%%%%%%%%%
\subsection{Second order sufficient conditions} 
%%%%%%%%%%%%%%%%%%%%%%%%%%%%%%%%%%%%%%%%%%%%%%

To apply the second order sufficient condition of the abstract Theorem \ref{SOSCCT}, we prove an additional estimate for the difference between adjoint states (see Lemma \ref{LemDeltap} below) and higher regularity of $\rho$ (see Theorem \ref{higher_reg1}). %

Before proceeding with the technical results, let us comment on a relevant point regarding the regularity of $F$. As correctly pointed out by one of the referees, by exploiting the fact that $D^2 F$ is locally Lipschitz continuous in our setting, one could dispense the use of Lemma \ref{LemDeltap} below and simplify some points of the proof of Proposition \ref{PropC3} that comes later. Nevertheless, for a sake of generality, we chose to keep the current approach for the present version of the manuscript, that holds for costs that are merely twice continuously differentiable.\\

We impose the following additional hypothesis, that is assumed to hold throughout the remainder of the article.
\begin{assumption}
\label{Hypb}
The function $b$ belongs to $W^{1,\infty}(\Omega;\cR^n)$ and, for a.a. $x \in \Gamma,$ it holds 
\[
\big(b(x) \otimes u\big) \cdot n(x) = 0 \quad  \text{for all } u \in \mathbb{R}^n.
\]
\end{assumption}
\begin{remark}
 Assumption \ref{Hypb} above is fulfilled in particular, if $b \in W^{1,\infty}_0(\Omega;\cR^n)$. 
 Moreover, it holds if $b(x) \cdot n(x) = 0$ a.e. on $\Gamma$ and the control is of the form $u(t)= \mathfrak{u}(t) \big(1,\dots,1\big) \in \cR^n,$ with $\mathfrak{u}$ being a scalar function.
 \end{remark}

Let us first prove the following technical result.

\begin{lemma} % Moved this property here. It is used for the first time in this section, in Proposition \ref{PropA}, to confirm (C2).
\label{LemDeltap}
Given $u_1,u_2\in \Uinfty$ with associated adjoint states $p_1,  p_2 \in W(0,T)$, respectively, we have
\be
\label{deltap}
\|p_2- p_1\|_{W(0,T)} \leq C\|p_1\|_{C([0,T];L^2(\Omega))}\|u_2-u_1\|_2,
\ee
where $C$ depends continuously on $\|u_2\|_\infty.$
\end{lemma}
\begin{proof}
In view of Proposition \ref{P5.1}, we have $p_1,  p_2 \in W(0,T)$, since $u_1,u_2\in \Uinfty$ is assumed.
Setting $\delta p:= p_2-p_1,$ we obtain
\be
\label{adjointdiff}
\begin{split}
- \partial_t \delta p - \nu \Delta \delta p+B[u_2]\cdot \nabla\delta p &=   {\alpha_Q}( \rho_2-\rho_1)  - \big({b \otimes (u_2- u_1)}\big)\cdot \nabla {p_1},\\
 \delta p(T) &= {\alpha_\Omega}( \rho_2(T)-\rho_1(T))\quad \text{on } \Omega,\\
 \partial_n(\delta p) &= 0\quad \text{on } \Sigma.
\end{split}
\ee
We can write the first two equations as
\begin{equation*}
\begin{split}
- \partial_t \delta p + a[u_2](\cdot,\delta p) &= f,\\
 \delta p(T) &={\alpha_\Omega}( \rho_2(T)-\rho_1(T)),
 \end{split}
\end{equation*}
where $f:= {\alpha_Q} (\rho_2-\rho_1) - \big(b \otimes (u_2- u_1)\big)\cdot \nabla {p_1}$.
For $\varphi \in \spaceV,$ we have
\[
\begin{split}
\iint_Q f \varphi dxdt &= \iint_Q   \Big(\alpha_Q(\rho_2-\rho_1) - \big(b \otimes (u_2- u_1)\big)\cdot \nabla {p_1}\Big) \varphi\,dxdt\\
&= \iint_Q  \Big( \alpha_Q(\rho_2-\rho_1)\varphi + p_1  {\rm div} \big(b \otimes (u_2- u_1)\varphi\big)\Big)  dxdt  \\
&\qquad -\iint_\Sigma  \varphi p_1 \big(b \otimes (u_2- u_1)\big) \cdot n ds dt\\
&= \iint_Q  \Big( \alpha_Q(\rho_2-\rho_1)\varphi + p_1  {\rm div} \big(b \otimes (u_2- u_1)\varphi\big)\Big)  dxdt  =: F(\varphi),
\end{split}
\]
where we used Assumption \ref{Hypb} in the third equality, and
where the functional $F$ belongs to $L^2(0,T;H^1(\Omega)^*)$. 
The term 
\[
q:= p_1 {\rm div} (b\otimes(u_2-u_1) \varphi)= p_1\big( \underbrace{{\rm div} (b\otimes(u_2-u_1))}_{L^2(0,T;L^\infty(\Omega))}\big)\varphi + p_1\underbrace{(b\otimes(u_2-u_1))}_{L^2(0,T;L^\infty(\Omega)^n)} \cdot \nabla \varphi.
\]
belongs to $L^1(Q)$, because we have  $p_1 \in C([0,T];L^2(\Omega))$. It can be estimated by
\[
\|q\|_{L^1(Q)} \le C\, \|p_1\|_{C([0,T];L^2(\Omega))} \|u_2-u_1\|_2 \|\varphi\|_{H^1(\Omega)^*}
\]
Now the claim follows easily from Propositions \ref{Propdeltarho} and \ref{P5.1} by estimating $F$.
\end{proof} 

For the remainder of the article, we additionally impose the following hypotheses, along with Assumption \ref{Hypb} introduced above.

\begin{assumption}[Requirements for higher regularity] \label{A6.1}
~

\begin{itemize}
\item[(i)]  The function $c$ has a potential $-V \in 
W^{2,\infty}(\Omega)$ so that 
\begin{equation} 
\label{potential}
c = \nabla V
\end{equation}
or $c$ belongs to $W^{1,\infty}(\Omega,\mathbb{R}^n)$ and satisfies
\begin{equation} 
\label{orthogonality}
c(x) \cdot n(x) = 0 \quad \mbox{a.e. on } \Gamma,
\end{equation}
\item[(ii)] the initial and the desired distributions $ \rho_0$ and $\rho_\Omega$ belong to $H^1(\Omega)$,
\item[(iii)]  and
$$
\min_{i=1,\dots,n} \gamma_i >0.
 $$
\end{itemize}
 \end{assumption}
 \begin{remark} The assumption \eqref{potential} that $c$ has a potential  was imposed by Breiten {\em et al.} \cite{BreitenKunischPfeiffer2018} to gain higher regularity of $\rho$.
\end{remark}
For smooth domains and a real-valued control, the next result follows from Breiten, Kunisch and Pfeiffer \cite[Proposition 6.1]{BreitenKunischPfeiffer2018}, who proved $W(0,T;H^2(\Omega),L^2(\Omega))$-regularity of $\rho$ by a Galerkin technique.
We extend the $C([0,T];H^1(\Omega))$-regularity of $\rho$ to bounded Lipschitz  domains
and to our setting of a vector-valued control. 
\begin{theorem}
\label{higher_reg1} 
If Assumptions \ref{Hypb} and \ref{A6.1} are fulfilled, 
 then, given any control $u\in L^\infty(0,T;\cR^n),$ the unique weak solution $\rho$ of the Fokker-Planck equation \eqref{FP},\eqref{initcond} belongs to $C([0,T];H^1(\Omega)).$
\end{theorem}

The proof can be found in the Appendix \ref{appendix}.

\begin{proposition}
\label{PropC3}
Under Assumptions \ref{Hypb} and \ref{A6.1}, problem \eqref{P} satisfies condition {\rm \ref{C3}}.
\end{proposition}

\begin{proof}
Consider a sequence $\{(u_k,v_k)\} \subset \Uad \times \Utwo$ with $u_k\to \ub$ in $\Utwo$ and $v_k \rightharpoonup v$ weakly in $\Utwo$ as in  {\rm \ref{C3}}. 
{For each $k,$ let $z_k$ be the linearized state associated to control $u_k$ in the direction $v_k.$}
We will proceed in several steps.

{\bfseries 1) Proof of \eqref{liminfD2J}.} 

{\em 1a)  Convergence of ${\{z_k\}}$.}
First, we show that $\{z_k\}$ has a well-defined limit. To this aim, we write the linearized equation for $z_k$
in the following form:
\begin{multline}
\label{lineareqk1}
\iint_Q \big( - z_k \partial_t \varphi + (\nu\nabla z_k+z_k B[\ub])\cdot \nabla \varphi \big) dxdt \\
= \iint_Q z_k b \otimes (\ub-u_k)\cdot \nabla \varphi \, dxdt - \iint_Q\rho_k (b\otimes v_k) \cdot \nabla \varphi \, dxdt,
\end{multline}
for all $\varphi \in W^{1,1}_2(Q)$ with $\varphi(\cdot,T)=0$. All controls $u_k$ belong to $\Uad$ and are therefore uniformly bounded in $\Uinfty$. Hence, thanks to Proposition \ref{PropG'}, the functions $z_k$ vary in a bounded set of $C([0,T];\spaceH)$. In view of this and since $u_k \to \ub$ in $\Utwo$, the term  $z_k b \otimes (\ub-u_k)$ under the first integral in the r.h.s. of \eqref{lineareqk1} tends to zero in $L^2(Q;\cR^n)$. Moreover, it is easy to confirm that
\[
\rho_k (b\otimes v_k) \rightharpoonup \bar\rho (b\otimes v) \mbox{ in } L^2(Q;\cR^n),\quad \text{ when } k \to \infty,
\]
since $\rho_k \to \bar \rho$ in $C([0,T];L^2(\Omega))$ in view of \eqref{estdeltarho}.
Therefore, the functional in the r.h.s. of \eqref{lineareqk1} generated by these terms converges weakly in $L^2(0,T;H^1(\Omega)^*)$. The solution operator associated with the differential operator of the l.h.s. is linear and continuous from $L^2(0,T;H^1(\Omega)^*)$
to $W(0,T)$, hence it preserves weak convergence, so that $\{z_k\}$ is weakly convergent to some $z$ in $W(0,T)$ that  satisfies the equation 
\be
\label{lineareqk2}
%\begin{split}
\iint_Q \left( - z \partial_t \varphi + (\nu\nabla z+zB[\ub])\cdot \nabla \varphi \right) dx\, dt
= - \iint_Q\bar\rho (b\otimes v) \cdot \nabla \varphi \, dx\,dt \quad 
%\end{split}
\ee
for all $\varphi$ mentioned above. Hence $z$ is the linearized state associated to $\bar\rho$ in the direction $v.$

In order to show \ref{C3}, let us take a subsequence $\{(u_{k_j},v_{k_j})\}$ such that
$$
\liminf_{k\to \infty} F''(u_k) v_k^2 = \lim_{j\to \infty} F''(u_{k_j})v_{k_j}^2.
$$
The corresponding subsequence $\{z_{k_j}\}$ converges weakly to $z$ in $W(0,T)$ and then, in view of Aubin-Lions' Lemma \cite{Aubin1963}, it contains a subsequence $\{z_{k_{j_\ell}}\}$ that converges strongly to $z$, i.e.
\be\label{aux1}
z_{k_{j_\ell}} \to z \quad \text{in } L^2(Q).
\ee
In order to simplify the notation, let us use the subindex $k$ for this subsequence.
One has
\be
\label{F2k}
F''(u_k) v_k^2 = \iint_Q \Big[ z_k^2-2\nabla p_k \cdot ( z_k  b\otimes v_k)\Big] dxdt + \sum_{i=1}^n \gamma_i\intT v_{k,i}^2 dt + \intO z_k(T)^2 dx,
\ee
for $z_k$  satisfying \eqref{lineareqk1}.

\vspace{5pt}

1b) Convergence of ${\{\rho_k\}}$ in $C([0,T];\spaceV)$.

Let us first discuss the case where $c$ obeys the condition \eqref{orthogonality} of Assumption \ref{A6.1}.
From Lemma \ref{Lmaxparabreg} applied to $\rho_k$, we have
\[
\|\rho_k\|_{W(0,T)} + \|\rho_k\|_{C([0,T];\spaceV)} \leq C \big( \|\div(\rho_k B[u_k])\|_{L^2(Q)} + \|\rho_0\|_{\spaceV}\big).
\]
Since $\{\rho_k\}$ is bounded in $W(0,T)$ and $\{u_k\}$ is bounded in $\Uinfty,$ the sequence $\{\div(\rho_k B[u_k])\}$ is bounded in $L^2(Q)$. 
Consequently, from previous display we deduce that $\{\rho_k\}$ is bounded in $C([0,T];\spaceV).$ For $\rho_k-\bar\rho$, we have the  equation
\[
\frac{\partial(\rho_k-\bar\rho)}{\partial t} - \nu \Delta(\rho_k-\bar\rho) =\underbrace{\div(\rho_k B[u_k])- \div (\bar\rho B[\ub])}_{\displaystyle\hat f_k}
\]
with initial condition $(\rho_k-\bar\rho)(0)=0$ and  homogeneous Neumann condition $\nabla(\rho_k-\bar\rho) \cdot n =0.$ For $\hat f_k$, we have
\begin{equation*}
\begin{split}
&\|\hat f_k\|_{L^2(Q)} =  \big\|\div(\rho_k B[u_k])- \div (\bar\rho B[\ub]) \pm \div(\rho_k B[\ub])\big\|_{L^2(Q)} \\
& \leq  \left\|\div\big(\rho_k (B[u_k]-B[\ub])\big)+ \div \big((\rho_k-\bar\rho) B[\ub]\big) \right\|_{L^2(Q)} \\
& \leq \sum_{i=1}^n \left[ \left( \left\| \frac{\partial \rho_k}{\partial x_i} \right\|_{C([0,T];\spaceH)} \|b_i\|_{L^\infty(\Omega)} 
+ \|\rho_k\|_{C([0,T];\spaceH)} \left\| \frac{\partial b_i}{\partial x_i} \right\|_{L^\infty(\Omega)} \right)\|u_{k,i}-\ub_i\|_2 \right.\\
&\,\,\,\,\,\,+\left. \left\|\frac{\partial (\rho_k-\bar\rho)}{\partial x_i}\right\|_{L^2(Q)} \|c_i+b_i\ub_{k,i}\|_{L^\infty(Q)} + \|\rho_k-\bar\rho\|_{L^2(Q)} \left\|\frac{\partial c_i}{\partial x_i}+\frac{\partial b_i}{\partial x_i} \ub_{k,i}\right\|_{L^\infty(Q)}\right].
\end{split}
\end{equation*}
Since $\{\rho_k\}$ is bounded in $C([0,T];\spaceV),$ $u_k \to \ub$ in $L^2(0,T;\cR^n)$ and $\rho_k \to \bar\rho$ in $W(0,T),$ the r.h.s. of latter display tends to 0, and hence $\|\hat f_k\|_{L^2(Q)}\to 0.$ We then have that $\rho_k \to \bar\rho$ in $C([0,T];\spaceV).$
% \deleted{From \eqref{estimatedeltapk} and \eqref{hatgk}, we get that $\delta p_k \to 0$ in $C([0,T];\spaceV)$ as desired.}

Now we briefly discuss the other case, where $c$ fulfils condition \eqref{potential} of Assumption \ref{A6.1}.
Here, we invoke again the transformation $w_k = \exp{(V/\nu)}\rho_k$ that we used to transform the Fokker-Planck equation to one with homogeneous boundary condition. Thanks to \eqref{hatf2}, the functions $w_k$ satisfy the linear heat equation \eqref{heateq} with r.h.s.
\begin{equation*} 
\label{hatf3}
\tilde f_k := - c \cdot \nabla w_k - \frac{1}{\nu} w_k c\cdot b\otimes u_k + {\rm div} (w_k b\otimes u_k).
\end{equation*}
It is obvious that $w_k$ converges to $\bar w = \exp{(V/\nu)}\bar \rho$ in $W(0,T)$. This yields
$\tilde f_k \to \bar f$ in $L^2(Q)$, where
\[
\bar f := - c \cdot \nabla \bar w - \frac{1}{\nu} \bar w c\cdot b\otimes \bar u + {\rm div} (\bar w b\otimes \bar u).
\]
Setting now $\hat f_k := \tilde f_k - \bar f$, we have $\hat f_k \to 0$ in $L^2(Q)$. In the same way as above, Lemma \ref{Lmaxparabreg} yields that $w_k - \bar w \to 0$ in  $C([0,T];\spaceV)$. This convergence transfers to the sequence $\rho_k - \bar \rho$. 

\vspace{5pt}

1c) Convergence of $\{p_k\}$ in $C([0,T];\spaceV)$.
 From the adjoint equation \eqref{adjoint}, we get, for a generic $p$ associated with some state-control pair $(\rho,u),$ with $u \in \Uinfty,$
\begin{equation}
\label{adjoint2}
\begin{split}
- \partial_t p - \nu \Delta p  &= {\alpha_Q}(\rho-\rho_Q)-B[u]\cdot \nabla p \quad \text{in } Q,\\
 p(T) &={\alpha_\Omega}(\rho(T)-\rho_\Omega) \quad \text{ in } \Omega,\\
 \partial_np &= 0 \quad \text{on } \Sigma.
\end{split}
\end{equation}
Since $u \in \Uinfty$, the r.h.s. ${\alpha_Q}(\rho-\rho_Q)-B[u]\cdot \nabla p$ belongs to $L^2(Q)$. Moreover, by Assumption \ref{A6.1}, (ii), ${\alpha_\Omega}( \rho(T)-\rho_\Omega)$ is in  $\spaceV.$ We can apply the regularity-Lemma \ref{Lmaxparabreg} and get that $\bar p$ and $p_k$ are in $C([0,T];\spaceV).$  Set $\delta p_k := p_k-\pb$. Then, from \eqref{adjointdiff}, we obtain, 
\be
\label{deltapk}
\begin{split}
- \frac{\partial\delta p_k}{\partial t} - \nu \Delta \delta p_k &=  \underbrace{{\alpha_Q}( \rho_k-\bar\rho)  - ({b \otimes (u_k- \bar u)}\big)\cdot \nabla \bar p - B[u_k]\cdot \nabla\delta p_k}_{\displaystyle \hat g_k}\quad \text{in } Q,\\
 \delta p_k(T) &= {\alpha_\Omega}\big( \rho_k(T) - \bar\rho(T) \big)\quad \text{on } \Omega,\\
 \partial_n(\delta p_k) &= 0\quad \text{on } \Sigma.
\end{split}
\ee
This is again a linear heat equation with homogeneous Neumann condition and r.h.s $\hat{g}_k$ in $L^2(Q)$.
Notice that all $u_k$ are uniformly bounded in $\Uinfty$ since they belong to $\Uad.$  Hence, Lemma \ref{Lmaxparabreg} applied to \eqref{deltapk} guarantees that $\delta p_k$  satisfies the estimate
\begin{equation}
\label{estimatedeltapk}
\|\delta p_k\|_{W(0,T)} + \|\delta p_k\|_{C([0,T];\spaceV)} \leq C\Big(\|\hat g_k\|_{L^2(Q)} + \| \rho_k(T) - \bar\rho(T)\|_{\spaceV}\Big).
\end{equation}
 For $\|\hat g_k\|_{L^2(Q)},$ we have
\begin{equation}
\label{hatgk}
\begin{split}
\|\hat g_k\|_{L^2(Q)} \leq &
\,\alpha_Q \|\rho_k-\bar\rho\|_{L^2(Q)} + \sum_{i=1}^n \|b_i\|_{L^\infty(\Omega)}\|u_{k,i}-\ub\|_2\left\| \frac{\partial \pb}{\partial x_i} \right\|_{C([0,T];\spaceH)}\\
&+\sum_{i=1}^n \big( \|c_i\|_{L^\infty(\Omega)}+\|b_i\|_{L^\infty(\Omega)}\|u_{k,i}\|_\infty\big)\left\| \frac{\partial \delta p_k}{\partial x_i} \right\|_{L^2(Q)}.
\end{split}
\end{equation}
From Lemma \ref{LemDeltap}, we know that $\nabla\delta p_k \to 0$ in $L^2(Q)^n$ since $u_k\to \ub$ in $\Utwo$ hence, the r.h.s. of latter display tends to 0 and, consequently, $\|\hat g_k\|_{L^2(Q)} \to 0.$ This fact, together with \eqref{estimatedeltapk} and the convergence $\rho_k \to \bar\rho$ in $C([0,T];\spaceV)$ proved in 1b) above, yield that $p_k \to \bar p$ in $C([0,T];\spaceV)$ as desired. This implies in particular
the convergence $\nabla p_k \to \nabla \bar p \quad \text{in } C([0,T];L^2(\Omega)^n)$ that is used in the next formula.

{\em 1d) Proof of \eqref{liminfD2J}.}  We are now ready to check that \eqref{liminfD2J} is satisfied.  Let us recall the expression of $F''(u_k)v_k^2$ given in \eqref{F2k}. For the bilinear term $\nabla p_k \cdot (z_k b\otimes v_k),$ we know from 1c) that $\nabla p_k \to \nabla \bar p \quad \text{in } C([0,T];L^2(\Omega)^n)$ and from \eqref{aux1} that $z_k \to z$ (strongly) in $L^2(Q)$. Moreover, we have $v_k \rightharpoonup v$ in $\Utwo.$ Therefore, the product $z_k b\otimes v_k$ converges weakly in $L^1(0,T;L^2(\Omega)^n)$ and
\begin{equation*}
\begin{split}
\iint_Q &\Big(\nabla p_k \cdot (z_k b\otimes v_k) - \nabla \pb \cdot (z b\otimes v) \Big)dx dt\\
&= \iint_Q \Big( \underbrace{(\nabla p_k-\nabla \pb)}_{C([0,T];L^2(\Omega)^n)} \cdot \underbrace{(z_k b\otimes v_k)}_{L^1(0,T;L^2(\Omega)^n)} \\
&\qquad+\, \nabla \pb \cdot \big((z_k-z) b\otimes v_k\big) + \nabla \pb \cdot (z b\otimes (v_k-v)) \Big)dx dt
\end{split}
\end{equation*}
tends to 0.

It remains to pass to the limit in the other terms of $F''(u_k)v_k^2.$ 
Notice that one has 
$$\iint_Q z_k^2 dxdt  \to \iint_Q z^2 dxdt$$ 
in view of the strong convergence $z_k \to z$ in $L^2(Q)$. Moreover, due to the weak convergence $v_k \rightharpoonup v$ in $L^2(0,T;\cR^n),$ one gets 
$$
\sum_{i=1}^n \gamma_i \intT v_i^2 dt \leq \liminf_{k\to \infty} \sum_{i=1}^n \gamma_i \intT v_{k,i}^2 dt.
$$ 
Since the mapping $z \mapsto z(\cdot,T)$ is linear and continuous from $W(0,T)$ to $L^2(\Omega),$ it is also weakly lower semicontinuous and, therefore, 
$$
\int_\Omega z(T)^2 dx \leq \liminf_{k\to \infty} \int_\Omega z_k(T)^2 dx.
$$
This concludes the proof of \eqref{liminfD2J}.

\vspace{7pt}

{\bfseries 2) Proof of \eqref{liminfD2J0}.}

It remains to show \eqref{liminfD2J0}. To this end, we assume that $v=0,$ this is $v_k  \rightharpoonup 0$ in $L^2(0,T;\cR^n).$ Thus, the corresponding limit $z$ of $\{z_k\}$ is also 0. We get
\be
\begin{split}
\liminf_{k\to \infty} F''(u_k)v_k^2 & = \liminf_{k\to \infty} \left[ \sum_{i=1}^n\gamma_i \intT v_{k,i}^2 dt + \intO z_k(T)^2 dx \right] \\
 &\geq \left( \min_{i=1,\dots,n} \gamma_i \right) \liminf_{k\to \infty} \|v_k\|^2_2. 
\end{split}
\ee  
This yields \eqref{liminfD2J0} for $\displaystyle\Lambda :=  \min_{i=1,\dots,n} \gamma_i > 0,$ and then the proof is concluded. 
\if{We have
\be
\iint_Q p_k{\rm div} ( z_k b \otimes v_k) dxdt = \intT\int_{\Gamma} p_kz_k b \otimes v_k \cdot n ds dt - \iint_Q \nabla p_k \cdot (b\otimes v_k)z_k dx dt.
\ee
}\fi
\end{proof}

Now we come to the main result of our paper.

\begin{theorem}[Second order sufficient condition]
\label{SOSC}
Let $\ub \in \Uad$ be such that the first order condition \eqref{FirstAbstract} is satisfied, and
\[
F''(\ub)v^2 >0,\quad \text{for all } v \in C(\ub) \backslash \{0\}.
\]
Under the assumptions \ref{Hypb} and \ref{A6.1}, there exist $\eps>0$ and $\delta >0$ such that the quadratic growth condition
\[
J(\ub) + \frac{\delta}{2} \|u-\ub\|_2^2 \leq F(u),\quad  \text{for all } u \in \Uad \cap B^2_\eps(\ub),
\]
holds, hence $\bar u$ is locally optimal in the sense of $L^2(0,T;\mathbb{R}^n)$.
\end{theorem}

\begin{proof}
The result  follows from the application of  Theorem \ref{SOSCCT} to our optimal control problem \eqref{P}.
The associated assumptions are fulfilled in view of the second order continuous differentiability of $F$ in $\Utwo$ and  Proposition \ref{PropC3}.
\end{proof}

\begin{remark}We should mention that also  \cite[Corollary 2.6]{CasasTroeltzsch2012} on local uniqueness of $\bar u$ and \cite[Theorem 2.7]{CasasTroeltzsch2012} regarding using an alternative critical cone can be directly transferred to our optimal control problem, since the required assumptions \ref{C1}-\ref{C3} are satisfied.
\end{remark}

%%%%%%%%%%%%%%%%%%%%%%%%%%%%%%%
\section{Appendix} \label{appendix}
%%%%%%%%%%%%%%%%%%%%%%%%%%%%%%%
\setcounter{equation}{0}

In this section, we prove the regularity result of Theorem \ref{higher_reg1}. First, 
we show that under the Assumptions \ref{Hypb} and \ref{A6.1} the state equation for $\rho$ can be transformed to a linear heat equation with homogeneous Neumann condition and right-hand side in $L^2(Q)$. This is the key for higher regularity of $\rho$.

To see this, we begin with the case where  Assumption \ref{Hypb} and condition \eqref{orthogonality} are fulfilled. Then,  for any control function $u \in L^\infty(0,T;\mathbb{R}^n),$  ${\rm div }(\rho(c+b\otimes u)) = {\rm div} (\rho B[u])\in L^2(Q)$ and
\[
- \int_\Omega \rho(c+ b\otimes u) \cdot  \nabla \varphi\, dx=\int_\Omega{\rm div }\big( \rho(c+b\otimes u) \big) \varphi \, dx  \quad \text{for all } 
\varphi \in H^1(\Omega),\, u \in \mathbb{R}^n,
\]
since  $\rho(c+b\otimes u) \cdot n = \rho B[u]\cdot n = 0$
on $\Gamma.$ Therefore, the state equation \eqref{FP},\eqref{initcond} reduces to 
the linear heat equation
\be \label{heateq}
\begin{split}
\partial_t \rho - \nu \Delta\rho & = \hat f,\\
\rho(0)&=\rho_0,\\
%(\nu\nabla\rho+\rho B[u]) \cdot n &=0,
\nabla\rho\cdot n &=0,
\end{split}
\ee
where
\begin{equation} \label{hatf1}
\hat f:= {\rm div} (\rho B[u]) \in L^2(Q).
\end{equation} 

Next, we confirm the reduction to an equation of the form \eqref{heateq}, if $c$ does not
satisfy \eqref{orthogonality} but fulfils \eqref{potential}. In this case, we follow an idea of 
\cite[proof of Proposition 2.1]{BreitenKunischPfeiffer2018}, and apply the transformation $w = \exp{(V/\nu)}\rho$. Notice that $w$ enjoys the same regularity as $\rho$, i.e. $w \in W(0,T)$. We find
\[
\nabla \rho = -\frac{1}{\nu} \exp{(-V/\nu)} \nabla V w +  \exp{(-V/\nu)}\nabla w=  \exp{(-V/\nu)} \, \left(-\frac{1}{\nu}c \, w +  \nabla w\right).
\]
Inserting this in the boundary condition of the state equation, we arrive at
\begin{equation*}
\begin{split}
&0=\nu \nabla \rho \cdot n + \rho B[u] \cdot n \\
&= \exp{(-V/\nu)} (-c \, w +  \nu \nabla w)\cdot {n} + \exp{(-V/\nu)}w (c + b\otimes u)\cdot n = \nu \exp{(-V/\nu)}\nabla w\cdot {n}
\end{split}
\end{equation*}
that holds true if, and only if, $\nabla w \cdot n= 0$ on $\Gamma$. Therefore, $w$ satisfies homogeneous Neumann boundary conditions. Moreover, inserting the transformation of $\rho$ in the state equation, we obtain, after some computations by the product rule applied to ${\rm div}\,\nabla  \big(\exp{(-V/\nu)}w)$ that $w$ satisfies the linear heat 
equation \eqref{heateq} with right-hand side
\begin{equation} 
\label{hatf2}
\hat f = - c \cdot \nabla w - \frac{1}{\nu} wc\cdot b\otimes u + {\rm div} (w b\otimes u).
\end{equation}
Since $w \in W(0,T)$ and $\nabla w \cdot n= 0$ on $\Gamma$, all summands above  belong to $L^2(Q)$.

Next, we provide a regularity result for equation \eqref{heateq} that can be found, for homogeneous Dirichlet boundary conditions and  a smooth boundary $\Gamma$, in  Evans \cite[Theorem 5, page 360]{Evans2010book}  or Ladyzhenskaya
{\em et al.} \cite{Ladyzhenskaya_etal1968}. For the convenience of the reader, we will prove the part on $C([0,T];H^1(\Omega))$-regularity in Lipschitz domains. It is difficult to find an associated reference, although the result is known. The proof was communicated to us by Joachim Rehberg (WIAS Berlin).

\begin{lemma} 
\label{Lmaxparabreg} 
If $\Omega$ is a bounded Lipschitz domain, $\hat f$ is an arbitrary function of $L^2(Q)$ and $\rho_0$ belongs to $H^1(\Omega)$,
then the unique weak solution $\rho$ of the linear heat equation \eqref{heateq} enjoys the higher regularity $\rho \in C([0,T];H^1(\Omega)).$
Moreover, the estimate
\be \label{a_priori_est}
\|\rho\|_{W(0,T)} + \|\rho\|_{C([0,T];H^1(\Omega))} \le C\, \big(\|\hat f\|_{L^2(Q)} + \|\rho_0\|_{H^1(\Omega)} \big)
\ee
is satisfied with some $C > 0$ not depending on $\hat f$ neither on $\rho_0$. 

\if{
\textcolor{green}{
If $\Omega$ is in addition of class $C^{1,1}$ or convex, then
the higher regularity $\rho \in W(0,T;H^2(\Omega),L^2(\Omega))$ is fulfilled, and the associated estimate
\[
\|\rho\|_{L^2(0,T;H^2(\Omega))} \le C\, (\|\hat f\|_{L^2(Q)} + \|\rho_0\|_{H^1(\Omega)})
\]
holds with some $C > 0$ not depending on $\hat f$ neither on $\rho_0$.}
}\fi
\end{lemma}

\begin{proof} The Laplace operator $\Delta$ is generated in $L^2(\Omega)$ by the bilinear form
\[
(\psi,\varphi) \mapsto  \int_\Omega \nabla \psi \cdot \nabla \varphi \, dx,
\]
defined on $H^1(\Omega) \times H^1(\Omega)$, cf. \cite[chpt. VI]{Kato1980} or \cite[chpt. 1]{Ouhabaz2005}. This operator, induced in 
$L^2(\Omega)$, is non-positive and self-adjoint, hence it generates an analytic semigroup in $L^2(\Omega)$.

It is known that the negative of any linear operator in a Hilbert space that generates an analytic semigroup has maximal parabolic regularity. In our setting, this means the following: if $S$  is an arbitrary bounded or unbounded real interval, for every $f \in L^2(S;L^2(\Omega))$ and each initial value $y_0$ from the real interpolation space $\big(L^2(\Omega),{\rm dom}_{L^2(\Omega)}\Delta\big)_{\frac{1}{2},2},$ there exists a unique $y \in W^{1,2}(S;L^2(\Omega)) \cap L^2(S;{\rm dom}_{L^2(\Omega)}\Delta)$ such that 
\[
\partial_t y -  \Delta y = f, \quad  y(0) = y_0.
\]
For this existence result, we refer to \cite[chpt. 1.3]{Ashyralyev_Sobolevski1994}, in particular to chpt. 1.3.3. Moreover, one has (see e.g. \cite[chpt. III.4.10]{Amann1995})
\be \label{embedding0}
W^{1,2}(S;L^2(\Omega)) \cap L^2(S,{\rm dom}_{L^2(\Omega)}\Delta) \hookrightarrow C\big(\bar{S}; (L^2(\Omega),{\rm dom}_{L^2(\Omega)}\Delta)_{\frac{1}{2},2}\big).
\ee

The operator $I -  \Delta$ is positive and self-adjoint with lower spectral bound $1$. Therefore, thanks to the spectral theory, its pure imaginary powers exist as bounded operators in $L^2(\Omega)$. In view of this, one has (see \cite[chpt. 1.18.10]{Triebel1978})
\[
\begin{aligned}
\big(L^2(\Omega),{\rm dom}_{L^2(\Omega)}\Delta\big)_{\frac{1}{2},2}& = \big(L^2(\Omega),{\rm dom}_{L^2(\Omega)}(-\Delta+I)\big)_{\frac{1}{2},2}\\
& = \big[L^2(\Omega),{\rm dom}_{L^2(\Omega)}(-\Delta+I)\big]_{\frac{1}{2}} = {\rm dom}_{L^2(\Omega)}
(-\Delta+I)^{\frac{1}{2}}.
\end{aligned}
\]
Since $\Delta$ is self-adjoint, the space ${\rm dom}_{L^2(\Omega)}
(-\Delta+I)^{\frac{1}{2}}$ is equal to its form-domain, i.e. $H^1(\Omega)$, \cite[chpt. VI 2.6]{Kato1980}.
Therefore, \eqref{embedding0} implies 
\be \label{embedding}
W^{1,2}(S;L^2(\Omega)) \cap L^2(S,{\rm dom}_{L^2(\Omega)}\Delta) \hookrightarrow C(\bar S;H^1(\Omega)).
\ee
Exactly the same arguments are true for $\nu \Delta$ instead of $\Delta$, except that the lower spectral
bound is now estimated from below by  $\nu$. These statements are valid in arbitrary domains $\Omega$.  

The estimate \eqref{a_priori_est} is shown as usual by considering a graph norm.
\end{proof}

\noindent {\bfseries Proof of Theorem \ref{higher_reg1}}: We have pointed out above that the state $\rho \in W(0,T)$ solves the linear heat equation \eqref{heateq} with a right-hand side $\hat f \in L^2(Q)$. Depending on the particular Assumption 
\ref{A6.1}-(i) on $c$, the function $\hat f$ is given by \eqref{hatf1} or by \eqref{hatf2}.  Now the claimed regularity follows from Lemma \ref{Lmaxparabreg}. \hfill $\Box$
%\end{proof}
\vspace{3ex}

\begin{remark} \label{Rem6.1} In addition to the continuous embedding \eqref{embedding}, also the embedding 
\[W^{1,2}(S;L^2(\Omega)) \cap L^2(S;{\rm dom}_{L^2(\Omega)}\Delta) \hookrightarrow C^{1/2}(\bar{S};L^2(\Omega))
\]
is continuous. Therefore, by interpolation arguments, cf. Triebel \cite{Triebel1978}  the same holds for the embedding
\[
W^{1,2}(S;L^2(\Omega)) \cap L^2(S;{\rm dom}_{L^2(\Omega)}\Delta) \hookrightarrow C^{\alpha}(\bar{S};H^{1/2-\alpha}(\Omega)),
\]
with $0 < \alpha < \frac{1}{2}$. In view of this, the embedding 
\[
W^{1,2}(S;L^2(\Omega)) \cap L^2(S;{\rm dom}_{L^2(\Omega)}\Delta) \hookrightarrow C(\bar{S};L^2(\Omega)),
\]
is compact. Therefore, for the heat equation \eqref{heateq}, the mapping $\hat f \mapsto \rho$ is compact from $L^2(Q)$ to $C([0,T];L^2(\Omega)) $.
\end{remark}

\section*{Acknowledgements}

We thank Joachim Rehberg (WIAS Berlin) for communicating to us the proof of Lemma 
\ref{Lmaxparabreg}. Moreover, we are grateful to Hannes Meinlschmidt (RICAM Linz) for his support in
proving the statements of Remark \ref{Rem6.1}.
We also acknowledge the referees for their useful and detailed remarks that helped us to essentially improve the manuscript.

\bibliographystyle{plain}
\bibliography{generalsoledad,fredi}

\begin{thebibliography}{10}

\bibitem{AlbiChoiFornasierKalise2017}
G.~Albi, Y.-P. Choi, M.~Fornasier, and D.~Kalise.
\newblock Mean field control hierarchy.
\newblock {\em Applied Mathematics \& Optimization}, 76(1):93--135, 2017.

\bibitem{AlbiPareschiZanella2014}
G.~Albi, L.~Pareschi, and M.~Zanella.
\newblock Boltzmann-type control of opinion consensus through leaders.
\newblock {\em Philos. Trans. R. Soc. Lond. Ser. A Math. Phys. Eng. Sci.},
  372(2028):20140138, 2014.

\bibitem{Amann1995}
H.~Amann.
\newblock {\em Linear and quasilinear parabolic problems. {V}ol. {I}},
  volume~89 of {\em Monographs in Mathematics}.
\newblock Birkh\"{a}user Boston, Inc., Boston, MA, 1995.
\newblock Abstract linear theory.

\bibitem{AnnunziatoBorzi2010}
M.~Annunziato and A.~Borz\`i.
\newblock Optimal control of probability density functions of stochastic
  processes.
\newblock {\em Math. Model. Anal.}, 15(4):393--407, 2010.

\bibitem{AnnunziatoBorzi2013}
M.~Annunziato and A.~Borz\`i.
\newblock A {F}okker--{P}lanck control framework for multidimensional
  stochastic processes.
\newblock {\em J. Comput. Appl. Math.}, 237(1):487--507, 2013.

\bibitem{AnnunziatoBorzi2018}
M.~Annunziato and A.~Borz{\`\i}.
\newblock A {F}okker--{P}lanck control framework for stochastic systems.
\newblock {\em EMS Surveys in Mathematical Sciences}, 5(1):65--98, 2018.

\bibitem{AronnaBonnansKroener2018}
M.S. Aronna, J.F. Bonnans, and A.~Kr\"{o}ner.
\newblock Optimal control of infinite dimensional bilinear systems: Application
  to the heat and wave equations.
\newblock {\em Math. Program.}, 168(1-2):717--757, 2018.

\bibitem{AronnaTroeltzsch2020v1}
M.S. Aronna and F.~Tr\"oltzsch.
\newblock First and second order optimality conditions for the control of
  {F}okker-{P}lanck equations (version 1).
\newblock 2020.
\newblock
  \href{https://arxiv.org/abs/2002.03988v1}{\texttt{[arxiv.org/abs/2002.03988v1]}}.

\bibitem{Ashyralyev_Sobolevski1994}
A.~Ashyralyev and P.~E. Sobolevski\u{\i}.
\newblock {\em Well-posedness of parabolic difference equations}, volume~69 of
  {\em Operator Theory: Advances and Applications}.
\newblock Birkh\"{a}user Verlag, Basel, 1994.
\newblock Translated from the Russian by A. Iacob.

\bibitem{Aubin1963}
J.-P. Aubin.
\newblock Un th{\'e}or{\`e}me de compacit{\'e}.
\newblock {\em CR Acad. Sci. Paris}, 256(24):5042--5044, 1963.

\bibitem{BoseTrimper2009}
T.~Bose and S.~Trimper.
\newblock Stochastic model for tumor growth with immunization.
\newblock {\em Phys. Rev. E (3)}, 79(5):051903, 2009.

\bibitem{BreitenKunischPfeiffer2018}
T.~Breiten, K.~Kunisch, and L.~Pfeiffer.
\newblock Control strategies for the {F}okker- {P}lanck equation.
\newblock {\em ESAIM: Control, Optimisation and Calculus of Variations},
  24(2):741--763, 2018.

\bibitem{Breiten_Kunisch_Pfeiffer2018_sicon}
T.~Breiten, K.~Kunisch, and L.~Pfeiffer.
\newblock Infinite-horizon bilinear optimal control problems: sensitivity
  analysis and polynomial feedback laws.
\newblock {\em SIAM J. Control Optim.}, 56(5):3184--3214, 2018.

\bibitem{cardaliaguet2010notes}
P.~Cardaliaguet.
\newblock Notes on mean field games.
\newblock Technical report, 2010.

\bibitem{CasasTroeltzsch2012}
E.~Casas and F.~Tr\"oltzsch.
\newblock Second order analysis for optimal control problems: Improving results
  expected from abstract theory.
\newblock {\em SIAM J. Optim.}, 22(1):261--279, 2012.

\bibitem{CasasWachsmuthWachsmuth2018}
E.~Casas, D.~Wachsmuth, and G.~Wachsmuth.
\newblock Second-order analysis and numerical approximation for bang-bang
  bilinear control problems.
\newblock {\em SIAM J. Control Optim.}, 56(6):4203--4227, 2018.

\bibitem{Chavanis2008nonlinear}
P.-H. Chavanis.
\newblock Nonlinear mean field {F}okker-{P}lanck equations. application to the
  chemotaxis of biological populations.
\newblock {\em Eur. Phys. J. B}, 62(2):179--208, 2008.

\bibitem{Chipot2012}
M.~Chipot.
\newblock {\em Elements of nonlinear analysis}.
\newblock Birkh\"auser, 2012.

\bibitem{DautrayLions1992}
R.~Dautray and J.-L. Lions.
\newblock {\em Mathematical analysis and numerical methods for science and
  technology: Evolution Problems I, volume 5}.
\newblock Springer Science \& Business Media, 1992.

\bibitem{DuanFornasierToscani2010}
R.~Duan, M.~Fornasier, and G.~Toscani.
\newblock A kinetic flocking model with diffusion.
\newblock {\em Comm. Math. Phys.}, 300(1):95--145, 2010.

\bibitem{Evans2010book}
L.C. Evans.
\newblock {\em Partial differential equations}.
\newblock American Mathematical Society, 2010.

\bibitem{FleigGuglielmi2017}
A.~Fleig and R.~Guglielmi.
\newblock Optimal control of the {F}okker-{P}lanck equation with
  space-dependent controls.
\newblock {\em J. Optim. Theory Appl.}, 174(2):408--427, 2017.

\bibitem{Furioli2017}
G.~Furioli, A.~Pulvirenti, E.~Terraneo, and G.~Toscani.
\newblock {F}okker--{P}lanck equations in the modeling of socio-economic
  phenomena.
\newblock {\em Math. Models Methods Appl. Sci.}, 27(01):115--158, 2017.

\bibitem{GiraultRaviart2012}
V.~Girault and P.-A. Raviart.
\newblock {\em Finite element methods for {N}avier-{S}tokes equations: theory
  and algorithms}, volume~5.
\newblock Springer Science \& Business Media, 2012.

\bibitem{GomesSaude2014}
D.A. Gomes and J.~Sa\'ude.
\newblock Mean field games models---a brief survey.
\newblock {\em Dyn. Games Appl.}, 4(2):110--154, 2014.

\bibitem{Herty2012}
M.~Herty, C.~J{\"o}rres, and A.N. Sandjo.
\newblock Optimization of a model {F}okker-{P}lanck equation.
\newblock {\em Kinet. Relat. Models}, 5(3), 2012.

\bibitem{Ioffe1979}
A.D. Ioffe.
\newblock Necessary and sufficient conditions for a local minimum 3: Second
  order conditions and augmented duality.
\newblock {\em SIAM J. Control and Optimization}, 17:266--288, 1979.

\bibitem{Kato1980}
T.~Kato.
\newblock {\em Perturbation theory for linear operators, Reprint of the corr.
  print of the 2nd edition}.
\newblock Classics in Mathematics. Springer-Verlag New York, Inc., New York,
  1980.

\bibitem{Ladyzhenskaya_etal1968}
O.~A. Lady\v{z}enskaja, V.~A. Solonnikov, and N.~N. Ural'ceva.
\newblock {\em Linear and quasilinear equations of parabolic type}.
\newblock Translated from the Russian by S. Smith. Translations of Mathematical
  Monographs, Vol. 23. American Mathematical Society, Providence, R.I., 1968.

\bibitem{LasryLions2007}
J.-M. Lasry and P.-L. Lions.
\newblock Mean field games.
\newblock {\em Jpn. J. Math.}, 2(1):229--260, 2007.

\bibitem{LioMag68a}
J.-L. Lions and E.~Magenes.
\newblock {\em Probl\`emes aux limites non homog\`enes et applications. {V}ol.
  1}.
\newblock Dunod, Paris, 1968.

\bibitem{Ouhabaz2005}
E.M. Ouhabaz.
\newblock {\em Analysis of heat equations on domains}, volume~31 of {\em London
  Mathematical Society Monographs Series}.
\newblock Princeton University Press, Princeton, NJ, 2005.

\bibitem{RoyAnnunziatoBorzi2016fokker}
S.~Roy, M.~Annunziato, and A.~Borz\`i.
\newblock A {F}okker--{P}lanck feedback control-constrained approach for
  modelling crowd motion.
\newblock {\em Journal of Computational and Theoretical Transport},
  45(6):442--458, 2016.

\bibitem{Ryzhik2018}
L.~Ryzhik.
\newblock Lectures notes (on mean field games).
\newblock Technical report, 2018.
\newblock
  \href{http://math.stanford.edu/~ryzhik/STANFORD/MEAN-FIELD-GAMES/notes-mean-field.pdf}{\texttt{link}}.

\bibitem{SchienbeinGruler1993}
M.~Schienbein and H.~Gruler.
\newblock Langevin equation, {F}okker-{P}lanck equation and cell migration.
\newblock {\em Bull. Math. Biol.}, 55(3):585--608, 1993.

\bibitem{Triebel1978}
H.~Triebel.
\newblock {\em Interpolation theory, function spaces, differential operators},
  volume~18 of {\em North-Holland Mathematical Library}.
\newblock North-Holland Publishing Co., Amsterdam-New York, 1978.

\bibitem{Troltzsch2010}
F.~Tr\"{o}ltzsch.
\newblock {\em Optimal control of partial differential equations. Theory,
  methods and applications}, volume 112 of {\em Graduate Studies in
  Mathematics}.
\newblock American Mathematical Society, Providence, RI, 2010.

\end{thebibliography}

\end{document}